\documentclass[12pt, reqno]{amsart}
\usepackage{amssymb}
\usepackage{mathtools}
\usepackage[english]{babel}
\usepackage{amscd}
\usepackage{amsmath}
\usepackage{enumerate}
\usepackage{ifpdf}
\usepackage{amsfonts}
\usepackage{dsfont}
\usepackage{amsthm}
\ifpdf
 \usepackage[novbox]{pdfsync}
 \usepackage[hyperfootnotes=false, pdfstartview={XYZ null null 1}, pdfpagemode=None]{hyperref}
 \hypersetup{hidelinks}
\fi

\newcommand{\di}{\,\mathrm{d}}
\numberwithin{equation}{section}

\usepackage{xcolor}

\long\def\REM#1{\relax}

\def\1#1{\overline{#1}}
\def\2#1{\widetilde{#1}}
\def\3#1{\widehat{#1}}
\def\4#1{\mathbb{#1}}
\def\5#1{\frak{#1}}
\def\6#1{{\mathcal{#1}}}

\newcommand{\UH}{\mathbb{H}}

\newcommand{\R}{\mathbb R}

\newcommand{\C}{\mathbb C}

\newcommand{\Hol}{{\sf Hol}}

\newcommand{\D}{\mathbb D}

\newcommand{\U}{{\mathfrak U}}
\newcommand{\Ut}[1]{{\mathfrak U\hskip.1em}_{#1}\hskip-.09em}
\newcommand{\UF}{{\mathfrak U\hskip.08em}}

\newcommand{\UD}{\mathbb{D}}

\newcommand{\mcite}[1]{\csname b@#1\endcsname}
\newcommand{\UC}{\mathbb{T}}

\theoremstyle{theorem}
\newtheorem {result} {Theorem}
\newcommand{\Real}{\mathbb{R}}
\newcommand{\Natural}{\mathbb{N}}
\newcommand{\Complex}{\mathbb{C}}
\newcommand{\ComplexE}{\overline{\mathbb{C}}}

\def\id{{\sf id}}

\newtheorem{theorem}{Theorem}[section]
\newtheorem{lemma}[theorem]{Lemma}
\newtheorem{proposition}[theorem]{Proposition}
\newtheorem{corollary}[theorem]{Corollary}

\theoremstyle{definition}
\newtheorem{definition}[theorem]{Definition}
\newtheorem{example}[theorem]{Example}

\theoremstyle{remark}
\newtheorem{remark}[theorem]{Remark}

\numberwithin{equation}{section}

\newcommand{\anglim}{\angle\lim}

\def\Re{\mathop{\mathtt{Re\hskip-.07ex}}}
\def\Im{\mathop{\mathtt{Im}}}
\newcommand{\Gen}{\mathsf{Gen}}
\newcommand{\GenDW}{\Gen_\tau}
\newcommand{\GenDWe}{\Gen'_\tau}
\let\le=\leqslant
\let\ge=\geqslant
\let\leq=\leqslant

\def\UC{\partial\UD}
\newcommand{\plw}{^{\textstyle\star}}
\newcommand{\zrw}{^\#}
\newcommand{\interior}{\mathop{\mathsf{int}}}
\newcommand{\extr}{\mathop{\mathsf{extr}}}
\newcommand{\Cara}{\mathcal C}
\newcommand{\Herg}{\mathcal P}
\newcommand{\ind}{\mathds{1}}

\title[Infinitesimal generators  with prescribed boundary fixed points]{Infinitesimal generators of semigroups with prescribed boundary fixed points}

\author[M. D. Contreras]{Manuel D. Contreras}
\author[S. D\'{\i}az-Madrigal]{Santiago D\'{\i}az-Madrigal}
\address{M. D. Contreras, S. D\'{\i}az-Madrigal: Camino de los Descubrimientos, s/n\\
Departamento de Matem\'{a}tica Aplicada~II and IMUS\\ Universidad de Sevilla\\ Sevilla,
41092\\ Spain.}\email{contreras@us.es} \email{madrigal@us.es}

\author[P. Gumenyuk]{Pavel Gumenyuk}\address{Pavel Gumenyuk: Department of
Mathematics\\ Milano Politecnico, via E. Bonardi 9\\ Milan 20133, Italy.}
\email{pavel.gumenyuk@polimi.it}

\date{\today}
\subjclass[2010]{Primary 37C10, 30C35; Secondary 30D05, 30C80, 37F99, 37C25}
\keywords{One-parameter semigroup; fixed point; boundary regular fixed point; infinitesimal generator; critical point; spectral value; value region; extreme point; Krein-Milman Theorem; Cowen-Pommerenke inequalities}

\thanks{Partially supported by the \textit{Ministerio
de Econom\'{\i}a y Competitividad} and the European Union (FEDER) PGC2018-094215-B-100 and by \textit{La Junta de Andaluc\'{\i}a}, FQM-133}

\begin{document}

\begin{abstract}
We study infinitesimal generators of one-parameter semigroups in the unit disk~$\UD$ having prescribed boundary regular fixed points. Using an explicit representation of such infinitesimal generators in combination with Krein\,--\,Milman Theory we obtain new sharp inequalities relating spectral values at the fixed points with other important quantities having dynamical meaning.
We also give a new proof of the classical Cowen\,--\,Pommerenke inequalities for univalent self-maps of~$\UD$.
\end{abstract}

\maketitle

\tableofcontents

\section{Introduction}
One-parameter semigroups of holomorphic self-maps of a hyperbolic simply connected domain in the complex plane~$\C$ constitute a classical topic in Complex Analysis, see e.g. \cite{Abate, BCD-Book, CD, EliShobook10, Shb}. They suite as a natural time-continuous analogue for discrete iteration of holomorphic maps and often appear in applications, for example in Probability Theory, see e.g. \cite{Goryainov-survey, TakahiroSebastian}, \cite[\S10.1]{Kyp}.

Thanks to the Riemann Mapping Theorem, one can restrict attention to one-parameter semigroups in the unit disk~$\UD:=\{z\in\C:|z|<1\}$.
From the dynamical point of view, an important role is played by the boundary fixed points, understood in the sense of angular limits. How presence of \textit{regular} boundary fixed points (see Def.\,\ref{DF_BRFP}) does affect behaviour of the self-map at internal points is the problem studied in many classical and recent works, e.g. \cite{MilnVas, BESh, CowenPommerenke, Frolova, Goryainov-paper, GumProkh} just to mention some. Boundary fixed points of one-parameter semigroups have been also subject to active interest, see e.g. \cite{CoDiPo04, CoDiPo06, Goryainov-Kudryavtseva, Gum14}.
In the present paper we study this problem for one-parameter semigroups with a prescribed finite set of boundary regular fixed points. Our main tools are the representation formula for infinitesimal generators of such one-parameter semigroups originally due to Goryainov~\cite{Goryainov-survey} (see Theorem~\ref{TH_representation-formula}) and Krein\,--\,Milman Theory, which turns out to be very useful in finding value regions of linear functionals on compact convex sets of holomorphic functions.
We obtain several sharp inequalities relating the spectral values at the boundary fixed points and at the Denjoy\,--\,Wolff point with the value of the infinitesimal generator at an internal point or other important dynamical characteristics of the semigroup.

The paper is organized as follows. In the next section, we collect some preliminary material, which we need to state the main results in Sect.\,\ref{S_mainresult}.
Further, in Sect.\,\ref{S_Herglotzfunctions} we obtain several auxiliary results concerning local boundary behaviour of holomorphic functions with non-negative real part.
Necessary results from Krein\,--\,Milman Theory are collected in Sect.\,\ref{S_KreinMilman}.

Sect.\,\ref{S_ValueRegions} contains our main results. First we establish (a more precise version of) Goryainov's representation formula mentioned above, see Sect.\,\ref{SS_representation}. In the next two subsections we state in a full detail our results on sharp value regions for such generators.
In particular, we take advantage of the fact that extreme points of the Carath\'eodory class are well-known, see Remark~\ref{RM_Caratheodory-class}. Going in a bit different direction, in the last part of Sect.\,\ref{S_ValueRegions} we study extremal points of two classes of infinitesimal generators of semigroups with prescribed boundary regular fixed points.

Finally, in Sect.\,\ref{S_param-and-CP-ineq} we give a new proof of the well-known Cowen\,--\,Pommerenke inequalities for univalent (i.e. injective holomorphic) self-maps of~$\UD$.
Three elementary statements used in Sect.\,\ref{S_param-and-CP-ineq} are proved in Appendix.

\section{Preliminaries}
\subsection{Discrete iteration}
For a domain $D\subset\Complex$, we denote by $\Hol(D)$ the linear space formed by all holomorphic functions from $D$ into $\Complex$. As usual, we  endow $\Hol(D)$ with the compact-open topology, which is the same as the topology of locally uniform convergence in~$D$. This turns $\Hol(D)$ into a locally-convex Hausdorff topological linear space.
For a set $E\subset\Complex$, we write $\Hol(D,E):=\{f\in\Hol(D)\colon f(D)\subset E\}$. Of particular interest in this paper, it will be the case when $D$ is the open unit disk $\UD$ and $f\in\Hol(\UD,\UD)$.

Thanks to the Schwarz Lemma, a holomorphic self-map $\varphi:\UD\to\UD$, $\varphi\neq\id_\UD$, can have at most one fixed point in~$\UD$. An important role is therefore played by the so-called boundary fixed points.
\begin{definition}\label{DF_BRFP}
A point~$\sigma\in\partial\UD$ is called a \textit{boundary fixed point} of a holomorphic self-map $\varphi:\UD\to\UD$ if the angular limit $\varphi(\sigma):=\anglim_{z\to\sigma}\varphi(z)$ exists and coincides with~$\sigma$. If in addition, the angular derivative $$\varphi'(\sigma):=\anglim_{z\to\sigma}\frac{\varphi(z)-\varphi(\sigma)}{z-\sigma}$$ exists finitely, then the boundary fixed point~$\sigma$ is said to be \textit{regular}. In what follows,  ``boundary regular fixed point'' will be abbreviated as \textit{``BRFP''}.
\end{definition}
It is known that at any boundary fixed point~$\sigma$ of a holomorphic self-map $\varphi:\UD\to\UD$, the angular derivative $\varphi'(\sigma)$ exists but it can be equal to~$\infty$; if $\varphi'(\sigma)$ is finite, then $\varphi'(\sigma)>0$ and moreover,
\begin{equation}\label{EQ-JuliasLemma}
\sup_{z\in\UD}\frac{1-|z|^2}{|z-\sigma|^2}\frac{|\varphi(z)-\sigma|^2}{1-|\varphi(z)|^2}=\varphi'(\sigma).
\end{equation}
see e.g. \cite[Proposition~4.13 on p.\,82]{Pombook92}. The latter statement is known as the \textit{Julia or Julia\,--\,Wolff Lemma}.

The classical \textit{Denjoy\,--\,Wolff Theorem} 
states that for any $\varphi\in\Hol(\UD,\UD)\setminus\{\id_\UD\}$ one of the two following alternatives holds:
\begin{itemize}
\item[(i)] either $\varphi$ has a unique fixed-point~$\tau\in\UD$, with $|\varphi'(\tau)|\le 1$,
\item[(ii)] or $\varphi$ is fixed-point free in~$\UD$, but it has a unique BRFP~$\tau\in\UC$ satifying~$\varphi'(\tau)\le 1$.
\end{itemize}
In both cases, $\tau$ is called the \textsl{Denjoy\,--\,Wolff point} of~$\varphi$, or in abbreviated form, the \textsl{DW-point}. In case~(i), the self-map $\varphi$ is said to be \textsl{elliptic}. In case~(ii), $\varphi$ is called \textsl{hyperbolic} or \textsl{parabolic} depending on whether ${\varphi'(\tau)<1}$ or ${\varphi'(\tau)=1}$.

For any BRFP $\sigma\neq\tau$,  we have $\varphi'(\sigma)>1$. By this reason, BRFPs different from the DW-point are often called \textsl{repelling}.

For these and more details concerning dynamics of holomorphic self-maps we refer an interested reader to \cite{Abate} and \cite[Sect. 1.8]{BCD-Book}.

\subsection{One-parameter semigroups}
Note that $\Hol(\UD,\UD)$ is a topological semigroup w.r.t. the composition. Continuous semigroup homomorphisms $ t\mapsto\phi_t\in\Hol(\UD,\UD)$  from the semigroup~$\big([0,+\infty),\,+\,\big)$ to $\Hol(\UD,\UD)$ are usually referred to as \textsl{one-parameter semigroups} (of holomorphic functions in~$\UD$).
In other words, a family $(\phi_t)_{t\ge0}\subset\Hol(\UD,\UD)$ is a one-parameter semigroup if $\phi_0=\id_\UD$, $\phi_s\circ\phi_t=\phi_{t+s}$ for any $s,t\ge0$, and $\phi_t(z)\to z$ as $t\to0^+$ for all~${z\in\UD}$. Note that the point-wise convergence leads to continuity in the open-compact topology because $\Hol(\UD,\UD)$ is a normal family.
Functions $\phi_t$ can be regarded as ``fractional iterates'' of~$\varphi:=\phi_1$.

We call a one-parameter semigroup~$(\phi_t)$ \textsl{non-trivial} if there is~$t>0$ such that $\phi_t\neq\id_\UD$. It is well known that for every non-trivial semigroup~$(\phi_t)$ all elements different from~$\id_\UD$ share the same Denjoy\,--\,Wolff point~$\tau$ and moreover, $\phi'_t(\tau)=e^{-\lambda t}$ for all $t\ge0$ and some~$\lambda\in\C$ with $\Re\lambda\ge0$ called the \textsl{spectral value of~$(\phi_t)$ \textsl(at its DW-point)}. In particular, every $\phi_t\neq\id_\UD$ has the same type and hence we can talk about elliptic, hyperbolic, and parabolic one-parameter semigroups.

Note also that if  $(\phi_t)$ is a non-trivial one-paremeter semigroup and $\phi_t=\id_\UD$ for some ${t>0}$, then all $\phi_t$'s are elliptic automorphisms of~$\UD$ and the semigroup~$(\phi_t)$ is just a non-Eucledean rotation around a common fixed-point in~$\UD$. In what follows, we will be mainly concerned with one-parameter semigroups whose elements have boundary fixed points and hence cannot not be elliptic automorphisms.

It turns out that every one-parameter semigroup is the semiflow of some holomorphic vector field in~$\UD$. Denote by $\UH$ the right half-plane ${\{z\colon \Re z>0\}}$. The following classical result is due to Berkson and Porta~\cite{BerPor78}, see also \cite[\S2]{FMS_ICAMI} and \cite[Sect. 10.1]{BCD-Book}.
\begin{result}\label{TH_BP}
Let $(\phi_t)$ be a one-parameter semigroup. Then for any~$z\in\UD$, $t\mapsto\phi_t(z)$ is differentiable, and there exists a unique $G\in\Hol(\UD)$ such that
\begin{equation}\label{EQ_flow}
\frac{\di \phi_t(z)}{\di t}=G(\phi_t(z))\quad \text{for all~$z\in\UD$ and all~$t\ge0$}.
\end{equation}
Moreover, if $(\phi_t)$ is non-trivial, then $G$ can be represented as
\begin{equation}\label{EQ_BP-formula}
G(z)=(\tau-z)(1-\overline\tau z)p(z),\quad z\in\UD,
\end{equation}
where $\tau\in\overline\UD$ is the DW-point of~$(\phi_t)$ and $p\in\Hol(\UD,\overline\UH \setminus\{0\})$.

Conversely, if~$G$ is given by~\eqref{EQ_BP-formula} with some $\tau\in\overline\UD$ and $p\in\Hol(\UD,\overline\UH{\setminus\{0\}})$, then there exists a unique non-trivial one-parameter semigroup~$(\phi_t)$ satisfying the ODE~\eqref{EQ_flow}.
\end{result}
\begin{definition}
The vector field~$G$ in the above theorem is called the \textsl{infinitesimal generator} of the one-parameter semigroup~$(\phi_t)$.
\end{definition}
The Berkson\,--\,Porta formula~\eqref{EQ_BP-formula} gives a necessary and sufficient condition for a holomorphic function~$G$ to be an infinitesimal generator.

Let us now discuss boundary fixed points of one-parameter semigroups.
\begin{definition}
A point $\sigma\in\UC$ is a \textsl{boundary regular fixed point} of a one-parameter semigroup~$(\phi_t)$, if $\sigma$ is a BRFP of~$\phi_t$ for all~$t\ge0$.
\end{definition}
\begin{remark}
In fact, the set of all BRFPs is the same for each~$\phi_t$ different from~$\id_\UD$, see e.g. \cite[Theorem 12.1.4]{BCD-Book} Hence ``for all~$t\ge0$'' in this definition can be replaced with ``for some~$t>0$ with~$\phi_t\neq\id_\UD$''.
\end{remark}

The following theorem characterizes BRFPs of one-parameter semigroups via their infinitesimal generators.
\begin{result}[\protect{\cite[Theorem~1]{CoDiPo06}}]\label{TH_BRFP-withPommerenke}
Let $(\phi_t)$ be a one-parameter semigroup in~$\UD$ and let $G$ be its infinitesimal generator.
Then $\sigma\in\partial\UD$ is BRFP of~$(\phi_t)$ if and only if the angular limit
$$
\lambda:=\anglim_{z\to\sigma}\frac{G(z)}{\sigma-z}
$$
exists finitely. In such a case, $\lambda\in\Real$ and $G'(z)\to-\lambda$ as $z\to\sigma$ non-tangentially. Moreover, $\phi_t'(\sigma)=e^{-\lambda t}$ for all~$t\ge0$.
\end{result}
\begin{definition}
The number~$\lambda$ in the above theorem is called the \textit{spectral value} of a one-parameter semigroup~$(\phi_t)$ \textsl{at a BRFP~$\sigma$.}
\end{definition}
Note that $\lambda<0$ for all BRFPs $\sigma$ different from the DW-point of~$(\phi_t)$.

\begin{remark}
Spectral values of a one-parameter semigroup at the DW-point and BRFPs can be interpreted, in a sense, as \textit{Lyapunov exponents} of fixed points of a dynamical system. If $\lambda$ is the spectral value, then $-\Re\lambda$ is the corresponding Lyapunov exponent.
\end{remark}

\section{Main results}\label{S_mainresult}
One {obvious and} natural question to study is what one can say about a one-parameter semigroup given information about its fixed points, e.g. position of its DW-point~$\tau$ and a finite set of (repelling) boundary regular fixed points, together with the corresponding spectral values. One of the basic quantities to look at is the velocity and direction of the trajectories, i.e. the value of the infinitesimal generator at a given point~$z_0\in\UD$. Using Moebius transformations, we can fix one of the points: e.g. we may suppose~$z_0=0$, keeping the DW-point~$\tau$ arbitrary.

The key tool is the following representation formula, see Theorem~\ref{TH_representation-formula}. Fix some negative numbers $\lambda_1,\ldots,\lambda_n$. A function $G:\UD\to\Complex$  is the infinitesimal generator of a one-parameter semigroup having the DW-point~$\tau\in\overline\UD$ and pairwise distinct boundary repelling fixed pints $\sigma_k$, $k=1,\ldots,n$, with the spectral values $\lambda'_k$ satisfying $\lambda_k\le\lambda'_k<0$ if and only if
\begin{equation}\label{EQ_mainresults-representation}
G(z)=(\tau-z)(1-\overline\tau z)\Big(p(z)\,+\,\sum_{k=1}^n\frac{|\tau-\sigma_k|}{2|\lambda_k|}\frac{\sigma_k+z}{\sigma_k-z}\Big)^{-1}
\end{equation}
for all~$z\in\UD$ and some $p\in\Hol(\UD)$ with $\Re p\ge0$.
Equality $\lambda'_k=\lambda_k$ holds if and only if $\anglim_{z\to\sigma_k} p(z)(1-\overline\sigma_k z)=0$. As we mentioned in the Introduction, formula~\eqref{EQ_mainresults-representation} in a bit weaker form was obtained earlier by Goryainov~\cite{Goryainov-survey}.

In Sect.\,\ref{S_ValueRegions}, using the above representation in combination with the Krein\,--\,Milman theory, see Sect.\,\ref{S_KreinMilman}, we find a number of (sharp) value regions relating~$G(0)$ with the repelling spectral values and local characteristics of~$G$ at the DW-point. In particular, if $G$ is the infinitesimal generator of a one-parameter semigroup $(\phi^G_t)$ with the DW-point $\tau\in\UD$, $\tau\neq0$, and BRFPs $\sigma_k\in\UC$, $k=1,\ldots,n$, then
\begin{equation}\label{EQ_G(0)}
\Re\frac{\tau}{G(0)}\ge A:=\sum_{k=1}^{n}\frac{|\tau-\sigma_{k}|^{2}}{2|\lambda_{k}|},
\end{equation}
\begin{multline}\label{Eq:pavel}
 \frac{1-|\tau|^2}{1+|\tau|^{2}}\left[ \Re\frac{\tau}{G(0)}-A\right]\le\\
\le ~\Re\frac{1-|\tau|^{2}}{\lambda}-\frac{1-|\tau|^2}{2}\sum_{k=1}^n\frac1{|\lambda_k|}~%
   \le~ \frac{1+|\tau|^2}{1-|\tau|^{2}}\left[ \Re\frac{\tau}{G(0)}-A\right],
\end{multline}
and
\begin{equation} \label{Eq:santiagoI}
    \left|\Im \left(\frac{1-|\tau|^{2}}{\lambda}-\frac{\tau}{G(0)}\right)- B\right|\leq \frac{2|\tau|}{1-|\tau|^{2}}\left[ \Re\frac{\tau}{G(0)}-A\right]\!\!,
    \end{equation}
{where $B:=\sum_{k=1}^{n}\frac{\Im (\overline \sigma_k\tau)}{|\lambda_{k}|}$ and  }$\lambda>0$, $\lambda_1,\ldots,\lambda_n<0$ are the spectral values of~$(\phi_t^G)$ at~$\tau$ and at $\sigma_1,\ldots,\sigma_n$, respectively. These inequalities are direct corollaries of Theorem~\ref{TH_value-region-interior}. Moreover, each inequality is sharp and the infinitesimal generators~$G$ for which equalities hold are completely characterized by Theorem~\ref{TH_value-region-interior}.

Similarly, for the case $\tau=0$, the following sharp inequality follows from Theorem\,\ref{TH_value-region-tau-zero}\,:
\begin{equation} \label{Eq:pavel2}
    \left|\frac{G''(0)}{2\lambda^{2}}-\sum_{k=1}^{n}\frac{\overline\sigma_k}{|\lambda_{k}|}\right|\leq  2\, \Re\frac{1}{\lambda}-\sum _{k=1}^{n}\frac{1}{|\lambda_{k}|}.
    \end{equation}

In the boundary case $\tau\in\UC\setminus\{\sigma_1,\ldots,\sigma_n\}$, inequality~\eqref{EQ_G(0)} holds as well, and for hyperbolic one-parameter semigroups we have
\begin{equation}
2\left(\frac{1}{\lambda}-\sum_{k=1}^n\frac1{|\lambda_k|}\right) \left[\Re\frac{\tau}{G(0)}-A\right]
\ge\left[\Re\frac{\tau}{G(0)}-A\right]^2+B^2,
\end{equation}
with $A$ and $B$ defined as above. This sharp inequality follows from Theorem~\ref{TH_value-region-boundary}.

Inequalities~\eqref{EQ_G(0)}, \eqref{Eq:pavel}, and \eqref{Eq:pavel2} imply that if $\tau\in\UD$, then
\begin{equation}\label{EQ_lambda1}
2\Re\frac{1}{\lambda}\ge\sum_{k=1}^n\frac1{|\lambda_k|}.
\end{equation}
Similarly, if $\tau\in\UC$, then according to Theorem~\ref{TH_value-region-boundary},
\begin{equation}\label{EQ_lambda2}
\frac{1}{\lambda}\ge\sum_{k=1}^n\frac1{|\lambda_k|}.
\end{equation}
The two inequalities above are sharp (see Corollary~\ref{CR_lambda}) and known (see \cite{ElShTa11}, \cite{CD}).

Finally, for parabolic one-parameter semigroups, i.e. for $\tau\in\UC$, $\lambda=0$, we obtain the following sharp inequality
\begin{equation}
0\le \anglim_{z\to\tau}\frac{(\tau-z)^3}{\tau^2G(z)}\le 2\Big(\Re\frac{\tau}{G(0)}-A\Big),
\end{equation}
see Theorem~\ref{TH_value-region-boundary2}.

Representation~\eqref{EQ_mainresults-representation} contains an arbitrary holomorphic function~$p$ with ${\Re p\ge0}$, which we call a Herglotz function. We take advantage of the fact  that the  Carath\'eodory class consisting of all Herglotz functions normalized by ${p(0)=1}$ is a compact convex cone in $\Hol(\UD)$ and that its set of its extreme points is well-known.
At the same time, in Sect.\,\ref{S_ValueRegions} we introduce another compact convex cone  $\GenDW(F,\Lambda)$ of infinitesimal generators of one-parameter semigroups with the DW-point~$\tau\in\overline\UD$ and given finite set $F$ of BRFPs.
In Theorem~\ref{TH_extereme-points-bis} we give a partial characterization of the extreme points of~$\GenDW(F,\Lambda)$. With the help of Krein\,--\,Milman Theorem in integral form, we recover the representation formula for infinitesimal generators in case of one given BRFP due to Goryainov and Kudryavtseva, see Corollary~\ref{CR_GorKudr}.

The compact convex cone $\GenDW(F)$ formed by all infinitesimal generators with the DW-point~$\tau\in\overline\UD$ and given finite set $F$ of BRFPs such that the spectral values $\lambda_k$ at $\sigma_k\in F$ satisfy $\sum_{k=1}^n|\lambda_k|\le 1$, is bit easier to study. We are able to obtain an explicit complete characterization of its extreme points, see Theorem~\ref{TH_extereme-points-tre}.

Thanks to a variant of Loewner's parametric represetation, see Sect.\,\ref{S_param-representation}, our results on infinitesimal generators can be used to obtain sharp estimates for univalent self-maps of~$\UD$, including those not embeddible in a one-parameter semigroup. As an illustration, in Sect.\,\ref{S_CP-ineq}, we give another proof of well-known inequalities due to Cowen and Pommerenke.

\section{Herglotz functions}\label{S_Herglotzfunctions}
The Berkson\,--\,Porta representation~\eqref{EQ_BP-formula}, which characterizes infinitesimal generators in~$\UD$, contains an arbitrary holomorphic function $p:\UD\to\Complex$ satisfying $\Re p\ge0$. We call such a function~$p$ a \textsl{Herglotz function}. Clearly, if a Herglotz function satisfies $\Re p(z)=0$ for some~$z\in\UD$, then $p$ is equal to a purely imaginary constant, and in this case $p$ is said to be a \textsl{trivial} Herglotz function.

In what follows we will need the following classical result, which is a version of Julia's Lemma for the half-plane.
Recall that by $\UH$ we denote the righ half-plane $\{z\colon \Re z>0\}$.

\begin{result}[\protect{see e.g. \cite[\S26]{Valiron:book}}]\label{TH_Julia-half-plane}
For any $f\in\Hol(\UH,\overline\UH)$, the limit
$$
f'(\infty):=\anglim_{\zeta\to\infty}\frac{f(\zeta)}{\zeta}
$$
exists finitely. Moreover,
$$
f'(\infty)=\inf_{\zeta\in\UH}\frac{\Re f(\zeta)}{\Re \zeta}\ge0.
$$
In particular, $$f(\zeta)=f'(\infty)\zeta+g(\zeta)\vphantom{\textstyle\int\limits_0^1}$$ for all $\zeta\in\UH$ and some $g\in\Hol(\UH,\overline\UH)$ satisfying $g'(\infty)=0$.
\end{result}

An imporant role in the present study is played by what we call contact points of Herglotz functions.

\begin{definition} Let $p$ be a Herglotz function. A point $\sigma\in\partial \D$ is called a \textit{contact point} for $p$ if the angular limit $p(\sigma):=\angle\lim_{z \to \sigma} p(z)$ exists and belongs to~$i\R$. Moreover, $\sigma$ is said to be a \textit{regular contact point} of $p$ if
\begin{equation}\label{EQ_zrw}
\angle\lim_{z \to \sigma}\frac{p(z)-p(\sigma)}{1-\overline{\sigma}z}
\end{equation}
exists finitely. If additionally $p(\sigma)=0$, we say that $\sigma$ is a \textit{regular zero} of $p$.
\end{definition}

For any Herglotz function~$p$ and any~$\sigma\in\UC$, we denote
\[
    p\zrw(\sigma):=
    \left\{
    \begin{array}{ll}
    \angle\lim_{z \to \sigma}\dfrac{p(z)-p(\sigma)}{1-\overline{\sigma}z}, & \textrm{if }\sigma \textrm{ is a regular contact point of } p, \\
    +\infty,& \textrm{otherwise}.
        \end{array}
    \right.
\]

\begin{remark}\label{RM_normal-functions}
Thanks to Montel's criterion, if $f\in\Hol(\UD)$ and $\C\setminus f(\UD)$ contains at least two distinct points, then $f$ is normal in~$\UD$, see \cite[\S9.1]{Pombook75}. According to the general version of Lindel\"of's Theorem due to Lehto and Virtanen (see, e.g., \cite[Theorem~9.3 on~p.\,268]{Pombook75}), if such a function~$f$ has a radial limit $\lim_{r\to1^-}f(r\sigma)\in\ComplexE$ at some~$\sigma\in\UC$, then it also has the angular limit at~$\sigma$. In particular, the angular limit $\anglim_{z\to\sigma}p(z)$ in the above definition can be replaced by the corresponding radial limit. Note that the angular limit in~\eqref{EQ_zrw} can be also replaced by the radial limit, but the reason for that is completely different. (Namely, one should use the Julia\,--\,Wolff\,--\,Carath\'eodory Theorem, see e.g. \cite[Sect. 1.7]{BCD-Book}.)
\end{remark}

\begin{remark}\label{RM_p-pole-zero}
Given a Herglotz function~$p$, according to Julia's Lemma in the half-plane (see Theorem~\ref{TH_Julia-half-plane}) applied to $\UH\ni\zeta\mapsto f(\zeta):=p\big(\sigma\frac{\zeta-1}{\zeta+1}\big)$, for any~$\sigma\in\UC$ the angular limit
$$
p\plw(\sigma):=\anglim_{z\to\sigma}(1-\overline \sigma z)p(z)
$$
exists and it is a non-negative real number.
Note that if $p$ is a non-trivial Herglotz function and $\sigma$ is a contact point of $p$, then $1/p$ and $1/(p-p(\sigma)$) are also (non-trivial) Herglotz functions. It follows that ${p\zrw(\sigma)\in[0,+\infty]}$, with and $p\zrw(\sigma)=0$ for some (resp. all) $\sigma\in\partial \UD$ if and only if $p$ is a trivial Herglotz function.
\end{remark}
\begin{remark}\label{RM_semicont}
Julia's  Lemma for the half-plane mentioned above tells us also that
\begin{equation}\label{EQ_Julia-inf-Re-Re}
p\plw(\sigma)=2\inf_{\Re\zeta>0}\frac{\Re p\big(\sigma\frac{\zeta-1}{\zeta+1}\big)}{\Re\zeta}
\end{equation}
for any Herglotz function~$p$ and any~$\sigma\in\UC$.
It follows that the map $p\mapsto p\plw(\sigma)$ is upper semicontinuous.
\end{remark}

\begin{theorem}\label{TH_regular-contact-point-of-HFunc}
\mbox{~}
\begin{itemize}
\item[(A)] For any Herglotz function~$p$ and any~$\sigma\in\UC$,
$$
p\zrw(\sigma)=\lim_{r\to1^-}\frac{\Re p(r\sigma)}{1-r}\in[0,+\infty].
$$
In particular, $\sigma$ is a regular contact point for~$p$ if and only if the limit in the right-hand side is finite.

\vskip1ex
\item[(B)] The map $p\mapsto p\zrw(\sigma)$ is lower semi-continuous on the cone in~$\Hol(\UD)$ formed by all Herglotz functions~$p$.
\end{itemize}
\end{theorem}
\begin{proof}
According Remark~\ref{RM_p-pole-zero}, if $\sigma$ is a regular contact point for~$p$, then $p\zrw(\sigma)$ is a non-negative number. Setting $z:=r\sigma$ and passing to the real part in~\eqref{EQ_zrw} yields assertion~(A) for the case of a regular contact point.

Assume now that $\sigma$ is \textsl{not}  a regular contact point for~$p$. We have to show that $\Re p(r\sigma)/(1-r)\to+\infty$ as $r\to1^-$.
Suppose on the contrary that $\Re p(r_n\sigma)\le M(1-r_n)$ for some constant $M\ge0$ and some sequence $(r_n)\subset[0,1)$ converging to~$1$. Clearly, we may suppose that $p$ is a non-trivial Herglotz function. Consider $f\in\Hol(\UD,\UD)$ defined by $f(z):=\big(p(z)-1\big)/\big(p(z)+1\big)$ for all~$z\in\UD$. By dropping a finite number of term in~$(r_n)$ we may suppose also that $\Re p(r_n\sigma)<1$ for all~$n\in\Natural$. Then
$$
 \frac{1-|f(r_n\sigma)|}{1-r_n}\le \frac{2\Re p(r_n\sigma)}{1-r_n}\le 2M\qquad\text{for all~$n\in\Natural$.}
$$
It follows that the boundary dilation coefficient of~$f$ at~$\sigma$ is finite and hence there exist finite limits
$$
 f(\sigma):=\anglim_{z\to\sigma}f(z) \quad\text{and}\quad f'(\sigma):=\anglim_{z\to\sigma}\frac{f(z)-f(\sigma)}{z-\sigma},
$$
with $\sigma\overline{f(\sigma)}f'(\sigma)>0$, see e.g. \cite[\S1.2.1]{Abate}. If $f(\sigma)\neq 1$, then we immediately see that $\sigma$ is a regular contact point for~$p$. If $f(\sigma)=1$, then it follows that there exists $$\lim_{r\to1^-}p(r\sigma)(1-r)=2/|f'(\sigma)|>0.$$ In both cases, our conclusions contradict the assumptions. This completes the proof of~(A).

To prove~(B) consider a sequence of Herglotz functions~$(p_n)$ converging locally uniformly in~$\UD$ to a Herglotz function~$p_0$ and such that $p_n\zrw(\sigma)$ tends to some $a\in[0,+\infty]$ as $n\to+\infty$. We have to show that ${p_0\zrw(\sigma)\le a}$.  Clearly, we may suppose that $a<+\infty$ and that $p_0$ is a non-trivial Herglotz function. Then, for all $n\in\Natural$ large enough, let us say for $n>n_0$,  $p_n$ is a non-trivial Herglotz function having a regular contact point at~$\sigma$ with $p_n\zrw(\sigma)<2a$.

Consider functions $q_n(z):=1/\big(p_n(z)-p_n(\sigma)\big)$, ${z\in\UD}$. Since $q_n\plw(\sigma)=1/p_n\zrw(\sigma)>1/(2a)$ for all~$n>n_0$, equality~\eqref{EQ_Julia-inf-Re-Re} in Remark~\ref{RM_semicont} for $p$ replaced by~$q_n$ implies that
$$
 \Re q_n\big(\sigma\tfrac{x-1}{x+1}\big)>\frac{x}{4a}\quad\text{for all~$x>0$ and all~$n>n_0$}.
$$
It follows that $p_n(r\sigma)\to p_n(\sigma)$ as $r\to 1^-$ \textit{uniformly} w.r.t.~$n>n_0$. Taking into account Remark~\ref{RM_normal-functions}, we see that $\sigma$ is a contact point of~$p_0$ and that $q_n\to {q_0:=1/\big(p_0-p_0(\sigma)\big)}$ locally uniformly in~$\UD$ as~$n\to+\infty$.
By Remark~\ref{RM_semicont},
$$
 1/a=\lim_{n\to+\infty}q_n\plw(\sigma)\le q_0\plw(\sigma)=1/p_0\zrw(\sigma).
$$
The proof is now complete.
\end{proof}

It is well-known, see e.g. \cite[\S1.9]{Duren:uni}, that $p$ is a Herglotz function if and only if it admits the following \textit{Riesz\,--\,Herglotz representation}
\begin{equation}\label{EQ_HerglotzRepresentation}
p(z)=\int\limits_{\UC}\frac{\varsigma+z}{\varsigma-z}\,\di\mu(\varsigma)~+~i\gamma\quad\text{for all~$z\in\UD$},
\end{equation}
where $\mu$ is a positive Borel measure on~$\UC$ and $\gamma\in\Real$. Moreover, the measure~$\mu$ is uniquely defined by~$p$, with $\mu(\UC)=\Re p(0)$ and  $\gamma=\Im p(0)$.
\begin{definition}
The measure $\mu$ in the above representation~\eqref{EQ_HerglotzRepresentation} will be referred to as the \textit{Herglotz measure} of~$p$.
\end{definition}
Using the one-to-one correspondence $p\mapsto (p-1)/(p+1)$ between Herglotz functions and $\Hol(\UD,\overline\UD\setminus\{1\})$, one can reformulate \cite[(VI-9) and~(VI-10)]{Sarason:sub-Hilbert} as follows.
\begin{lemma}\label{LM_Sarason}
Let $p$ be a Herglotz function and $\mu$ the Herglotz measure of~$p$. Then
\begin{equation}\label{EQ_Sarason}
p\zrw(\sigma)=2\int_{\UC}|\varsigma-\sigma|^{-2}\,\di\mu(\varsigma)\quad\text{for any $\sigma\in\UC$.}
\end{equation}
In particular, $\sigma$ is a regular contact point of~$p$ if and only if $\varsigma\mapsto|\varsigma-\sigma|^{-2}$ is $\mu$-integrable.
\end{lemma}
For completeness, we provide a direct proof using the same idea as in~\cite{Sarason:sub-Hilbert}.
\begin{proof}
Formula~\eqref{EQ_HerglotzRepresentation} easily implies that for any~$r\in(0,1)$,
\begin{equation}\label{EQ-Re-part}
\frac{\Re p(r\sigma)}{1-r}\,r^2=(1+r)\int_{\UC}\frac{r^2}{|\varsigma-r\sigma|^2}\,\di\mu(\varsigma),
\end{equation}
which tends to $2\int_{\UC}|\varsigma-\sigma|^{-2}\,\di\mu(\varsigma)$ as $r\to1^-$ by Levi's Monotone Convergence Theorem.  Note that this argument is valid both in case of finite and infinite value of the integral in the r.h.s. and hence it simply remains to apply Theorem~\ref{TH_regular-contact-point-of-HFunc}\,(A).
\end{proof}

\begin{remark}
It is known~\cite[Proposition~4.7 on p.\,79]{Pombook92}, see also the proof of \cite[Theorem~10.5, pp.\,305--306]{Pombook75}, that a contact point~$\sigma$ of a Herglotz function~$p$ is regular if and only if $p'$ has a finite angular limit at~$\sigma$. Therefore, in view of Lemma~\ref{LM_Sarason}, it might be plausible to expect that $\sigma\in\UC$ is a contact point (not necessarily regular) if and only if $\varsigma\mapsto |\varsigma-\sigma|^{-1}$ is integrable w.r.t. the Herglotz measure of~$p$. However, as we show below, see Lemma~\ref{LM_contact-point} and Example~\ref{EX_noconverse}, the latter condition is sufficient but not necessary for~$\sigma$ to be a contact point.
\end{remark}

\begin{lemma}\label{LM_contact-point}
Let $p$ be a Herglotz function and $\mu$ the Herglotz measure of~$p$. Let $\sigma\in\UC$. If $\varsigma\mapsto|\varsigma-\sigma|^{-1}$ is~$\mu$-integrable, then $\sigma$ is a contact point of~$p$ and
\begin{equation}\label{EQ_contact-point}
p(\sigma)=\int_{\UC}\frac{\varsigma+\sigma}{\varsigma-\sigma}\di\mu(\varsigma) + i\Im p(0)= i\int_0^{2\pi}\cot(\theta/2)\di\mu(\sigma e^{i\theta}) + i\Im p(0).
\end{equation}
\end{lemma}
\begin{proof}
The hypothesis implies that $\mu(\{\sigma\})=0$. Bearing this in mind, the elementary observation that $|\varsigma-\sigma|\le |(\varsigma/r)-\sigma|$ for any~$\varsigma\in\UC$ and any~$r\in(0,1)$ allows us to pass to the limit in
$$
p(r\sigma)=\frac{1}{r}\int_{\UC}\frac{\varsigma+r\sigma}{(\varsigma/r)-\sigma}\di\mu(\varsigma) + i\Im p(0)
$$
as $r\to1^-$ with the help of Lebesgue's Dominated Convergence Theorem. With Remark~\ref{RM_normal-functions} taken into account, this shows that $p(\sigma):=\anglim_{z\to\sigma}p(z)$ exists finitely  and proves the first equality in~\eqref{EQ_contact-point}. Writing~$\theta:={\arg(\overline\sigma\varsigma)}$ and separating the real and imaginary parts of the integrand leads us to the second equality in~\eqref{EQ_contact-point}. In particular, ${p(\sigma)\in i\,\Real}$, and the proof is complete.
\end{proof}

\begin{example}\label{EX_noconverse}
Let us construct a positive Borel measure $\mu$ on~$\UC$ such that $|\varsigma+1|^{-1}$ is not $\mu$-integrable but the Herglotz function $p$ given by~\eqref{EQ_HerglotzRepresentation} has a contact point at~$\sigma=-1$. Restricting consideration to measures with $\mu(\{1\})=0$ and using relation between the Herglotz\,--\,Riesz representation and Nevanlinna's representation, see e.g. \cite[p.\,135--139 and eq.\,(V.42)]{Bhatia}, we reduce the problem to finding a positive Borel measure~$\nu$ compactly supported on~$\R$ such that $1/|t|$ is not $\nu$-integrable, but the function
$$
P(z):=\int_\R\frac{\di\nu(t)}{t-z},\quad \Im z>0,
$$
tends to a real number as $z:=iy\to0$, $y>0$. Note that $\Im P\ge0$ and hence $P$ is a normal function, see e.g. \cite[pp.\,261-262]{Pombook75}. Therefore, by a theorem of Lindel\"of, see e.g. \cite[Theorem~9.3 on p.\,268]{Pombook75}, the existence of the limit as $z\to0$ along the imaginary axis implies existence of the angular limit of~$P$ at~$0$.

Consider the measure $\nu$ defined by
$$
\di\nu(t):=\frac{\,\ind_{[-1/e,\,1/e]}(t)}{\log(1/|t|)}\,\di t,\quad t\in\Real,
$$
where $\ind_{\!A}$ stands for the indicator function of a set~$A\subset\Real$. Clearly, $\int_\R|t|^{-1}\di \nu(t)=+\infty$.

Since $\nu$ is invariant w.r.t. the transformation $t\mapsto-t$, $\Re P(iy)=0$ for all~$y>0$. Therefore, with the help of the variable change $u:=t/y$, for any $y\in(0,1)$ we have
\begin{align*}
0\le \frac{P(iy)}{2i}&=\int\limits_{0}^{1/e}\frac{y\di t}{(t^2+y^2)\log(1/t)}=\int\limits_0^{1/(ey)}\frac{\di u}{(u^2+1)\log(1/(uy))}\\[1.5ex]
&=\int\limits_0^{1/(e\sqrt y)}\frac{\di u}{(u^2+1)\log(1/(uy))}~+\int\limits_{1/(e\sqrt y)}^{1/(ey)}\frac{\di u}{(u^2+1)\log(1/(uy))}\\[1.5ex]
&\le\frac1{1+\frac12\log(1/y)}\int\limits_0^{+\infty}\frac{\di u}{u^2+1}~+\int\limits_{1/(e\sqrt y)}^{+\infty}\frac{\di u}{u^2+1}.
\end{align*}
Both summands in the last line tend to zero as~$y\to0^+$. Hence $\lim_{y\to0^+}P(iy)=0\in\Real$ as desired.
\end{example}

\section{Extreme points and Krein--Milman Theorem}\label{S_KreinMilman}
Throughout this section, by $X$ we will denote a locally-convex Hausdorff topological linear space and
$\extr K$ will stand for the set of all extreme points of a set~$K\subset X$.

\begin{result}[Krein\,--\,Milman Theorem, see e.g. \protect{\cite[\S18.1.2]{Kadec}}]\label{TH_KreinMilman}
Let  $K\subset X$ be a non-empty convex compact set.
Then $\extr K\neq\emptyset$ and moreover, $K$ coincides with the closure of the convex hull of~$\extr K$.
\end{result}

\begin{remark}\label{RM_affine}
Let $X$ and $K$ be as in the Krein\,--\,Milman Theorem and $L:X\to\C^m$ a continuous affine map. Then the image $L(K)$ is clearly a compact convex set in $\C^m$. The preimage $L|_{K}^{-1}(z)$ of any extreme point~$z$ of~$L(K)$ is a face for~$K$ and hence, by the Krein\,--\,Milman Theorem, it contains an extreme point. In fact, $L|_{K}^{-1}(z)$ has to contain at least two distinct extreme points unless it is a singleton. Therefore, $\extr L(K)\subset L(\extr K)$. It follows that $L(K)$ coincides with the convex hull of $L(\extr K)$. Moreover, if $L$ is injective on $\extr K$, then each $z\in\extr L(K)$ has exactly one preimage w.r.t.~$L|_K$, which is, of course, an extreme point of~$K$.
\end{remark}

It is well known, see e.g. \cite[\S18.1.2]{Kadec}, that for any continuous linear real-valued functional~$L$, the maximum of~$L$ over a non-empty convex compact set~$K\subset X$ is attained at an extreme point of~$K$.
The standard argument can be easily adjusted to extend this assertion to upper-semicontinuous functionals. More precisely, the following theorem holds.
\begin{result}\label{TH_max-on-extreme}
Let $K\subset X$ be a non-empty convex compact set and let $L$ a (densely-defined) linear real-valued functional on~$X$. If $L$ is defined everywhere on~$K$ and $L|_K$ is upper-semicontinuous, then
$$
 \max_{p\in K}L(p)=\max_{p\in\extr K}L(p).
$$
Moreover, if the maximum of $L$ is attained on~$\extr K$ at a unique point~$x_0$, then the same holds for the whole set~$K$, i.e. $L(x)<L(x_0)$ for all~$x\in K\setminus\{x_0\}$.
\end{result}

\begin{remark}\label{RM_Caratheodory-class}
We will apply the above results to various classes of holomorphic functions with positive real part. Among them is the \textit{Carath\'eodory class~$\Cara$,} which consists of all Herglotz functions~$p$ normalized by~$p(0)=1$. This class has been thoroughly studied. In particular, it is known that $\Cara$ is a convex compact subset of $\Hol(\UD)$ and its  extreme points form a one-parameter family, namely $\extr\Cara=\{q_\sigma\colon\sigma\in\UC\}$, $q_\sigma(z):={(\sigma-z)/(\sigma+z)}$ for all~$z\in\UD$, see e.g. \cite{Holland}.
\end{remark}

Krein\,--\,Milman Theorem applies to \textsl{compact} convex sets. We will have to study \textsl{non-compact} subclasses of the Carath\'eodory class~$\Cara$. Suitable extension of the Krein--Milman theory is given in~\cite{WW, Winkler}, where (infinite-dimensional) simplices of probability measures are considered instead of compact convex sets in a topological vector space. For our purposes, the very general setting of~\cite{WW, Winkler} is not necessary. The following two theorems are, in fact, corollaries of the indicated results for the special case of Borel probability measures on the unit circle~$\UC$. We denote the set of all such measures by $P$.

\begin{result}[\protect{\cite[Theorem 2.1 and Example~2.1]{Winkler}}]\label{TH_W-extr}
Fix $n\in\Natural$. Let $f_1,\ldots, f_n$ be real Borel functions on~$\UC$ and $c_1,\ldots,c_n\in\Real$. If $\mu_0$ is an extreme point of
$$
H_1:=\Big\{\mu\in P\colon \text{\small $f_j$ is $\mu$-integrable and $\int_{\UC} f_j(\sigma)\,\di\mu(\sigma)=c_j$ for all $j=1,\ldots,n$}\Big\}
$$
or an extreme point of
$$
H_2:=\Big\{\mu\in P\colon \text{\small $f_j$ is $\mu$-integrable and $\int_{\UC} f_j(\sigma)\,\di\mu(\sigma)\le c_j$ for all $j=1,\ldots,n$}\Big\},
$$
then $\mu_0$ is a convex combination of at most $n+1$ Dirac measures on~$\UC$.
\end{result}

\begin{result}[\protect{\cite[Theorem~3.2 and Proposition~3.1]{Winkler}}]\label{TH_W-est}
Let $H:=H_1$ or $H:=H_2$, where $H_1$ and $H_2$ are defined as in Theorem~\ref{TH_W-extr}. Let $g$ be a real Borel function on~$\UC$ such that for any~$\mu\in H$ the integral $$I(\mu):=\int_{\UC} g(\sigma)\,\di\mu(\sigma)$$ exists  with values in $[-\infty,+\infty]$. Then
\begin{equation}\label{EQ_W-est}
\sup\big\{I(\mu)\colon \mu\in H\big\}=\sup\big\{I(\mu)\colon \mu\in \extr H\big\}.
\end{equation}
\end{result}

\section{Infinitesimal generators of one-parameter semigroups with given boundary regular fixed points}
\label{S_ValueRegions}
Let $\Gen$ stand for the class of all infinitesimal generators in~$\UD$. For $G\in\Gen$, we denote by~$(\phi_t^G)$ the corresponding one-parameter semigroup in~$\UD$. We will write~$G\in\GenDW$ to specify that $\tau$ is the DW-point of~$(\phi_t^G)$, adopting the useful convention that the trivial generator~$G\equiv 0$ belongs to~$\GenDW$ for any~$\tau\in\overline\UD$.  Moreover,  $\lambda(G)$ stands for the spectral value of~$(\phi^G_t)$ at its DW-point~$\tau$.

As before, let $F:={\{\sigma_1,\sigma_2,\ldots,\sigma_n\}} {\subset\UC}$, with $\sigma_j\neq\sigma_k$
for~$j\neq k$.  For $\tau\in\overline\UD\setminus F$ and $\Lambda:=(\lambda_1,\lambda_2,\ldots,\lambda_n)\in(-\infty,0)^n$, denote by $\GenDW(F,\Lambda)$ the class of all infinitesimal generators $G\in\GenDW$ such that $(\phi_t^G)$ has BRFPs~$\sigma_1,\ldots,\sigma_n$ with the spectral values $\lambda'_1,\ldots,\lambda'_n$, respectively, subject to the inequalities $\lambda_k\le\lambda'_k\le0$, $k=1,\ldots,n$. Note that if at least one of $\lambda'_k$'s vanishes, then $G\equiv0$. Finally, let us denote by $\GenDWe(F,\Lambda)$ the subclass of $\GenDW(F,\Lambda)$ in which the spectral values are \textsl{exactly} $\lambda_1,\ldots,\lambda_n$.

For $F$ and $\Lambda$ introduced above and $\tau\in\overline\UD\setminus F$, we denote
\begin{equation}\label{EQ_p0-alpha_k}
p_0(z;F,\Lambda):=\sum\limits_{k=1}^n\alpha_k\dfrac{\sigma_k+z}{\sigma_k-z},\quad\text{where~}~ \alpha_k:=\frac{|\tau-\sigma_k|^2}{2|\lambda_k|},~k=1,\ldots,n.
\end{equation}
Furthermore, we will write $\interior A$ for the interior of a set~$A$.

\subsection{Representation formula}\label{SS_representation}
The following theorem establishes a representation formula for the classes~$\GenDW(F,\Lambda)$ and $\GenDWe(F,\Lambda)$.
\begin{theorem}\label{TH_representation-formula}
For $F$, $\tau$, and $\Lambda$ introduced above, the following two statements hold.
\begin{itemize}
\item[(A)] A function $G:\UD\to\Complex$ belongs to~$\GenDW(F,\Lambda)\setminus\{G\equiv0\}$ if and only if there exists a Herglotz function~$p$ such that
\begin{equation}\label{EQ_representation-formula}
G(z)=\frac{(\tau-z)(1-\overline\tau z)}{p(z)+p_0(z;F,\Lambda)}\quad \text{for all~$z\in\UD$},
\end{equation}
where $p_0$ is given by~\eqref{EQ_p0-alpha_k}.\vskip1ex

\item[(B)] Similarly, $G:\UD\to\Complex$ belongs to~$\GenDWe(F,\Lambda)$ if and only if~\eqref{EQ_representation-formula} holds with some Herglotz function $p$ satisfying $p\plw(\sigma_k)=0$ for all~$k=1,\ldots,n$.
\end{itemize}
\end{theorem}

\noindent Originally, the representation formula~\eqref{EQ_representation-formula} is due to Goryainov, see \cite[Theorem~3]{Goryainov-survey}. We present here a quite different proof because our version of this important result is more precise. Namely, we state explicitly the relation between the parameters $\lambda_k$'s in the r.h.s. of~\eqref{EQ_representation-formula} and the spectral values of~$G$, which is not mentioned in~\cite{Goryainov-survey}. The latter aspect is important for the rest of Sect.\,\ref{S_ValueRegions}.
In our proof, we will use the following lemma.

\begin{lemma}\label{LM_Julia}
Let $q$ be a Herglotz function and $\sigma\in\UC$. Then
$$
 q(z):=p(z)\,+\,\frac{q\plw(\sigma)}{2}\,\frac{\sigma+z}{\sigma-z}
$$
for all $z\in\UD$ and some Herglotz function~$p$ with~$p\plw(\sigma)=0$.
\end{lemma}
\begin{proof}
Apply Theorem~\ref{TH_Julia-half-plane}\, for $f(\zeta):=q\big(\sigma\tfrac{\zeta-1}{\zeta+1}\big)$, $\zeta\in\UH$, and notice that ${f'(\infty)=q\plw(\sigma)/2}$.
\end{proof}

\begin{proof}[{\fontseries{bx}\selectfont Proof of Theorem \ref{TH_representation-formula}}]
Suppose first that $G$ is given by~\eqref{EQ_representation-formula} with some Herglotz function $p$. Since $p_0(\cdot;F,\Lambda)$ is a non-trivial Herglotz function, $p_*(z):=\big(p(z)+p_0(z;F,\Lambda)\big)^{-1}$, $z\in\UD$, is also a Herglotz function.
Thanks to the Berkson\,--\,Porta formula~\eqref{EQ_BP-formula}, it follows that $G\in\GenDW$. Therefore, according to Theorem~\ref{TH_BRFP-withPommerenke}, in order to show that $G\in\GenDW(F,\Lambda)$, we have to check that
\begin{equation}\label{EQ_G-lim}
\anglim_{z\to\sigma_k}\frac{G(z)}{z-\sigma_k}\le|\lambda_k|,\quad k=1,\ldots,n.
\end{equation}
By means of elementary computations we see that limit in the left hand side exists finitely as long as $p_*\zrw(\sigma_k)$ is finite and that
$$
\anglim_{z\to\sigma_k}\frac{G(z)}{z-\sigma_k}=|\tau-\sigma_k|^2 p_*\zrw(\sigma_k)=\frac{|\tau-\sigma_k|^2}{p\plw(\sigma_k)+2\alpha_k}=\frac{|\lambda_k|}{1+p\plw(\sigma_k)/(2\alpha_k)}.
$$
Hence we may conclude that $G\in\GenDW(F,\Lambda)$, and clearly $G\not\equiv0$. Moreover, if $p\plw(\sigma_k)=0$, $k=1,\ldots,n$, then we see that $G\in\GenDWe(F,\Lambda)$.

To prove the converse statements, suppose that $G\in\GenDW(F,\Lambda)$ and that $G\not\equiv0$. Then by the Berkson\,--\,Porta formula~\eqref{EQ_BP-formula},  $G(z)=(\tau-z)(1-\overline\tau z)p_*(z)$ for all~$z\in\UD$, where $p_*$ is a Herglotz function. Clearly, $p_*\not\equiv 0$ and hence $q:=1/p_*$ is also a Herglotz function.

By Theorem~\ref{TH_BRFP-withPommerenke}, $G$ satisfies condition~\eqref{EQ_G-lim}, meaning in particular that the limit in the left hand side  exists finitely. It follows that $q\plw(\sigma_k)\ge 2\alpha_k$, with the equalities occurring for all~$k=1,\ldots,n$ if~$G\in\GenDWe(F,\Lambda)$.

Note also that for $q_\sigma(z):=(\sigma+z)/(\sigma-z)$, $\sigma\in\UC$, we have $q_\sigma\plw(\sigma')=0$ for all ${\sigma'\in\UC\setminus\{\sigma\}}$. Taking this into account, we can apply Lemma~\ref{LM_Julia} repeatedly to obtain
$$
 q(z)=p_1(z)+\sum_{k=1}^n\beta_k\frac{\sigma_k+z}{\sigma_k-z}\quad\text{for all $z\in\UD$},
$$
where $\beta_k:=q\plw(\sigma_k)/2\ge\alpha_k$ and $p_1$ is a Herglotz function such that $p_1\plw(\sigma_k)=0$ for all~$k=1,\ldots,n$. As a result we have
$q(z)=p(z)+p_0(z;F,\Lambda)$, where $p(z):=p_1(z)+\sum_{k=1}^n(\beta_k-\alpha_k)(\sigma_k+z)/(\sigma_k-z)$ is a Herglotz function, which coincides with~$p_1$ provided $G\in\GenDWe(F,\Lambda)$. This immediately leads to representation~\eqref{EQ_representation-formula} and completes the proof of~(A) and~(B).
\end{proof}

\begin{corollary}\label{CR_compact}
The class $\GenDW(F,\Lambda)$ is a compact convex subset of~$\Hol(\UD)$. The class $\GenDWe(F,\Lambda)$ is a convex dense subset of~$\GenDW(F,\Lambda)$.
\end{corollary}

\begin{proof}
The convexity of $\GenDW(F,\Lambda)$ and $\GenDWe(F,\Lambda)$ follow easily from the Berkson\,--\,Porta representation~\eqref{EQ_BP-formula} and Theorem~\ref{TH_BRFP-withPommerenke}.

To prove that  $\GenDW(F,\Lambda)$ is compact, consider a sequence $(G_n)$ contained in this class. Clearly, passing to a subsequence, we may suppose that $G_n\not\equiv0$ for all~$n\in\Natural$. By Theorem~\ref{TH_representation-formula}\,(A) there exists a sequence of Herglotz functions~$(p_n)$ such that $G_n(z)=(\tau-z)(1-\overline\tau z)/\big(p_n(z)+p_0(z;F,\Lambda)\big)$ for all~$z\in\UD$ and all~$n\in\Natural$. Since Herglotz functions form a normal family in~$\UD$, passing to a subsequence we may suppose that $p_n\to p_*$ locally uniformly in~$\UD$ as $n\to+\infty$, where $p_*$ is either a Herglotz function or~$p_*\equiv\infty$. In the latter case, $G_n\to0$ locally uniformly in~$\UD$ as $n\to+\infty$. Note that by the very definition $G_*(z):=0$, $z\in\UD$, belongs to $\GenDW(F,\Lambda)$.  In the former case, $$G_n(z)\to G_*(z):=(\tau-z)(1-\overline\tau z)/\big(p_*(z)+p_0(z;F,\Lambda)\big)$$ locally uniformly in~$\UD$. By Theorem~\ref{TH_representation-formula}\,(A), $G_*\in \GenDW(F,\Lambda)$. Thus, every sequence $(G_n)\subset\GenDW(F,\Lambda)$ has a subsequence converging in~$\Hol(\UD)$ to an element of~$\GenDW(F,\Lambda)$, i.e. $\GenDW(F,\Lambda)$ is compact.

It remains to show that $\GenDWe(F,\Lambda)$ is dense in $\GenDW(F,\Lambda)$. To this end, fix $G\in\GenDW(F,\Lambda)$ and write, using Theorem~\ref{TH_representation-formula}\,(A), $G(z)=(\tau-z)(1-\overline\tau z)/\big(p(z)+p_0(z;F,\Lambda)\big)$. By Theorem~\ref{TH_representation-formula}\,(B), for any~$n\in\Natural$ the function
$$
G_n(z):=(\tau-z)(1-\overline\tau z)/\big(p(r_n z)+p_0(z;F,\Lambda)\big),\quad r_n:=1-1/n,
$$
belongs to~$\GenDWe(F,\Lambda)$. Clearly, $G_n\to G$ locally uniformly in~$\UD$. The proof is complete.
\end{proof}

\subsection{Elliptic semigroups}\label{SS_valueReg-ellipt}
Combining the representation for~$\GenDW(F,\Lambda)$ established in the previous section with the Krein\,--\,Milman Theory, see Sect.\,\ref{S_KreinMilman}, we are going to study this class quantitatively.
First we consider the case~$\tau\in\UD$. Using the notation introduced at the beginning of Sect.\,\ref{S_ValueRegions}, we can state our results as follows.

\begin{theorem}\label{TH_value-region-interior}
Let $\tau\in\UD\setminus\{0\}$. The value region
$$
V_\tau(F,\Lambda):=\Big\{\big(G(0),\lambda(G)\big)\in\C^2\colon G\in\GenDW(F,\Lambda)\Big\}
$$
coincides with the set $$
 \Big\{(\zeta,\omega)\in\Complex^2\colon\zeta\in Z,\,\omega\in\Omega_\zeta\Big\},\quad\text{where~}
  ~Z:=\Big\{\zeta\in\C\colon \big|{\textstyle\big(2\tau^{-1}\sum\limits_{k=1}^n\alpha_k}\big)\,\zeta\,-\,1\big|\le1\Big\}
$$
and $\,\Omega_\zeta$ is the closed disk (which degenerates to a point when $\zeta\in\partial Z$) given by
\begin{gather}\label{EQ_Omega(zeta)}
\Omega_\zeta:=\Big\{\omega\in\C\colon\displaystyle\Big|\frac{1-|\tau|^2}{\omega}-a_\zeta\Big|\le \frac{2\,|\tau|}{1-|\tau|^2}\Re \ell_\zeta\Big\},\\\notag
\ell_\zeta:=\frac{\tau}{\zeta}-{\textstyle\sum\limits_{k=1}^n\alpha_k},\quad
\displaystyle a_\zeta:=\frac{1+|\tau|^2}{1-|\tau|^2}\Re \ell_\zeta+i\Im \ell_\zeta + p_0(\tau;F,\Lambda),
\end{gather}
for all $\zeta\in Z$ except for $\zeta=0$, in which case $\Omega_\zeta=\Omega_0:=\{0\}$.
\vskip.75ex

Each boundary point of~$V_\tau(F,\Lambda)$ is delivered by a unique $G\in\GenDW(F,\Lambda)$ of the form
\begin{gather}
G(z)=G_{\zeta,\sigma}(z):=\frac{(\tau-z)(1-\overline\tau z)}%
{\displaystyle \frac{\sigma+z}{\sigma-z}\Re \ell_\zeta+i\Im \ell_\zeta+ p_0(z;F,\Lambda)},\quad\zeta\in\interior Z,~\sigma\in\UC,\label{EQ_G-zeta,sigma}\\
\intertext{or} G(z)=G_\zeta(z):=\frac{(\tau-z)(1-\overline\tau z)}%
{\displaystyle i\Im \ell_\zeta+ p_0(z;F,\Lambda)},\quad\zeta\in\partial Z\setminus\{0\},\label{EQ_G-zeta}
\end{gather}
or~$G(z)=G_0(z)\equiv0$.
\end{theorem}

\begin{proof}
According to Theorem~\ref{TH_representation-formula}\,(A), $G\in\GenDW(F,\Lambda)\setminus\{G\equiv0\}$ if and only if
\begin{equation}\label{EQ_GenDW-iff}
G(z)=(\tau-z)(1-\overline \tau z)\big/\big(p(z)+p_0(z;F,\Lambda)\big),\quad z\in\UD,
\end{equation}
for some Herglotz function~$p$. The inequality $\Re p(0)\ge0$ is equivalent to $\zeta:=G(0)\in Z\setminus\{0\}$.
From~\eqref{EQ_GenDW-iff} we immediately get ${p(0)=\ell_\zeta}$. In particular, if ${\zeta\in\partial Z\setminus\{0\}}$, then ${\Re \ell_\zeta=0}$; hence $p(z)\equiv i\Im\ell_\zeta$ and $G$ is given by~\eqref{EQ_G-zeta}, with $$\lambda(G)=-G'(\tau)={(1-|\tau|^2)}\big/{\big(i\Im\ell_\zeta+p_0(\tau;F,\Lambda)\big)}.$$ It is elementary to check that in this case, $\Omega_\zeta$ given by~\eqref{EQ_Omega(zeta)} is the singleton consisting of precisely this point.

Note also that $G(0)=0$ implies $G\equiv0$ because $\tau\neq0$. Therefore, it remains to consider the case $G(0)\in\interior Z$. To this end fix $\zeta\in\interior Z$ and solve the problem to find the range~$\Omega_\zeta$ of the map $G\mapsto\lambda(G)$ on the set
$
\mathcal G_\zeta:=\big\{G\in\GenDW(F,\Lambda)\colon G(0)=\zeta\big\},
$
which is described by formula~\eqref{EQ_GenDW-iff} with $p(z):=q(z)\Re\ell_\zeta+i\Im\ell_\zeta$, where $q\in\Cara$. Therefore, our task reduces to finding the range $R_\zeta$ of the affine map $\Cara\ni q\mapsto1/\lambda(G)\in\Complex$. According to Remark~\ref{RM_affine}, it is sufficient to find values of the map at the extreme points of~$\Cara$, which are well-known, see Remark~\ref{RM_Caratheodory-class}. The extreme points of~$\Cara$ correspond to the infinitesimal generators~\eqref{EQ_G-zeta,sigma} and we see that $R_\zeta$ is the convex hull of $\{1/\lambda(G_{\zeta,\sigma})\colon \sigma\in\UC\}$. By means of elementary  computations, this leads to the conclusion that the range of $G\mapsto\lambda(G)$ on
$\mathcal G_\zeta$ coincides with the closed disk~$\Omega_\zeta$ defined by~\eqref{EQ_Omega(zeta)}, with the boundary points delivered by the infinitesimal generators~\eqref{EQ_G-zeta,sigma}. Note that $q\mapsto1/\lambda(G)$ is injective on~$\extr\Cara$. Therefore, by Remark~\ref{RM_affine}, each $\omega\in\partial\Omega_\zeta$ is the image of \textit{exactly one} $G\in\mathcal G_\zeta$. The proof is complete.
\end{proof}

\vskip1ex

For the case of~$\tau=0$, we clearly have $G(0)=0$ for all~$G\in\GenDW(F,\Lambda)$, but we can choose another functional, e.g.~$G\mapsto G''(0)$.

%

\begin{theorem}\label{TH_value-region-tau-zero}
In the above notation, the value region
$$
\widehat V(F,\Lambda):=\Big\{\big(G''(0),\lambda(G)\big)\in\C^2\colon G\in\Gen_0(F,\Lambda)\Big\},
$$
coincides with the set $\Big\{(\zeta,\omega)\in\Complex^2\colon\omega\in \Omega,\,\zeta\in Z_\omega\Big\}$, where
\begin{gather*}
    \Omega:=\{\omega\colon|\omega-r|\le r\},\quad\text{with $~r:=\textstyle\big(\sum_{k=1}^n|\lambda_k|^{-1}\big)^{-1}$,}\\
    Z_\omega:=\Big\{\zeta\colon\left|\dfrac\zeta{2\omega^2}-\textstyle\sum\limits_{k=1}^n\dfrac{\overline\sigma_k}{|\lambda_k|}\right| \le 2\Re\dfrac1\omega-\sum_{k=1}^n|\lambda_k|^{-1}\Big\}\text{ for all~$\omega\in\Omega\setminus\{0\}$}
\end{gather*}
and $Z_0:=\{0\}$.
\vskip.75ex

Each boundary point of~$\widehat V(F,\Lambda)$ is delivered by a unique point of the form
\begin{equation}
    G(z)=\widehat G_{\omega,\sigma}(z):=-\frac{z}{\dfrac{\sigma+z}{\sigma-z}\Re\hat\ell_\omega\,+\,i\Im\hat\ell_\omega\,+\,p_0(z;F,\Lambda)},\quad \omega\in\interior\Omega,~\sigma\in\UC,
\end{equation}
where $\hat\ell_\omega:=1/\omega-(1/2)\sum_{k=1}^n|\lambda_k|^{-1}$, or
\begin{equation}\label{EQ_G-omega}
    G(z)=\widehat G_\omega(z):=-\frac{z}{i\Im\,\hat\ell_\omega+\,p_0(z;F,\Lambda)},\quad \omega\in\partial\Omega\setminus\{0\},
\end{equation}
or $G(z)=\widehat G_0(z)\equiv0$.
\end{theorem}

\begin{proof}
According to Theorem~\ref{TH_representation-formula}\,(A), $G\in\Gen_0(F,\Lambda)\setminus\{G\equiv0\}$ if and only if
\begin{equation*}\label{EQ_GenDW-iff-zero}
G(z)=-z\big/\big(p(z)+p_0(z;F,\Lambda)\big),\quad z\in\UD,
\end{equation*}
with some Herglotz function $p$. Using this representation (and bearing in mind that~$\tau=0$) we immediately obtain
\begin{gather}\label{EQ_tau-zero_lambda}
\lambda(G)=-G'(0)=\Big(p(0)+\frac12\sum_{k=1}^{n}|\lambda_k|^{-1}\Big)^{-1},\\
\label{EQ_tau-zero_Gsecond}
\frac{G''(0)}{2\lambda(G)^2}=p'(0)+\sum_{k=1}^{n}\frac{\overline\sigma_k}{|\lambda_k|}.
\end{gather}
From~\eqref{EQ_tau-zero_lambda} it immediately follows that the range of $\Gen_0(F,\Lambda)\ni G\mapsto \lambda(G)$ is exactly the closed disk $\Omega$ defined in the statement of the theorem and moreover, every boundary point $\omega\in\partial\Omega$ corresponds to exactly one $G\in\Gen_0(F,\Lambda)$, namely, $G=\widehat G_\omega$ defined by~\eqref{EQ_G-omega}. It is therefore, elementary to check the statement of the theorem for this case.

Now fix some $\omega\in\interior \Omega$ and find the range~$Z_\omega$ of $G\mapsto G''(0)$ over all $G\in\Gen_0(F,\Lambda)$ satisfying~$\lambda(G)=\omega$. According to~\eqref{EQ_tau-zero_Gsecond}, our problem is reduced to finding the range of $p\mapsto p'(0)$ over all Herglotz functions $p$ with $p(0)=\hat\ell_\omega$. The rest of the proof consists of using the representation $p(z)=q(z)\Re\hat\ell_\omega\,+\,i\Im\hat\ell_\omega$, $q\in\Cara$, along with Remarks~\ref{RM_affine} and~\ref{RM_Caratheodory-class}, and some elementary computations. The details are similar to those in the proof of Theorem~\ref{TH_value-region-interior} and therefore we omit them.
\end{proof}


\subsection{Non-elliptic semigroups}\label{SS_valueReg-boundary} Now let us consider  the boundary case $\tau\in\UC$. We start with the analogue of Theorem~\ref{TH_value-region-interior}. The proof of Theorem~\ref{TH_value-region-boundary} is based on the following result, which can be of some independent interest.

\begin{proposition}\label{PR_in-Cara-class}
For~$a\in\Real$ and~$\tau\in\UC$,
$$
\min\big\{q\zrw(\tau)\colon q\in\Cara~\text{\small has a contact point at~$\tau$ with $q(\tau)=ia$}\big\}~=~\frac{1+a^2}2.
$$
The minimum is attained for the unique function $q(z)=q_{\sigma}(z):=(\sigma+z)/(\sigma-z)$, where $\sigma:=-\tau(1+ia)/(1-ia)$.
\end{proposition}

\begin{proof}
Denote
$$
\Cara(\tau,a):=\{q\in\Cara\colon\text{\small has a contact point at~$\tau$ with $q(\tau)=ia$}\big\}.
$$
According to Lemmas~\ref{LM_Sarason} and~\ref{LM_contact-point}, the problem to find the sharp lower bound for~$q\zrw(\tau)$ in~$\Cara(\tau,a)$ is equivalent to finding the sharp lower bound for
$$
I(\mu):=2\int_{\UC}|\varsigma-\tau|^{-2}\,\di\mu(\varsigma)~\in\,(0,+\infty]
$$
over all Borel probability measures on~$\UC$ subject to the conditions that $\varsigma\mapsto|\varsigma-\tau|^{-1}$ is~$\mu$-integrable  and that
\begin{equation}\label{EQ_cond-isoperim}
\int_0^{2\pi}\!\!\cot(\theta/2)\di\mu(\tau e^{i\theta})~=~a.
\end{equation}
Note that this set of probability measures corresponds via the Riesz\,--\,Herglotz representation~\eqref{EQ_HerglotzRepresentation} to a proper subset $\Cara^*(\tau,a)$ of $\Cara(\tau,a)$, which however contains all~$q\in\Cara(\tau,a)$ with finite~$q\zrw(\tau)$.

Theorems~\ref{TH_W-est} and~\ref{TH_W-extr} allows us to restrict ourselves to probability measures supported
at one or two points on~$\UC$. The corresponding elements of $\Cara$ have the form $q:=\lambda q_1+(1-\lambda)q_2$, where $\lambda\in[0,1]$ and $q_j(z):=(\sigma_j+z)/(\sigma_j-z)$, $z\in\UD$, $j=1,2$, with some~$\sigma_1,\sigma_2\in\UC$. Denote $a_j:=-iq_j(\tau)\in\Real$, $j=1,2$. Then $a=-iq(\tau)=\lambda a_1+(1-\lambda)a_2$ and $q\zrw(\tau)=\lambda f(a_1)+(1-\lambda)f(a_2)$, where $f(x):=(1+x^2)/2$.
Thanks to the fact that $f$ is strictly convex, for any~$\lambda\in[0,1]$ and any $a_1,a_2\in\Real$,
$$
\lambda f(a_1)+(1-\lambda)f(a_2)\ge f\big(\lambda a_1+(1-\lambda)a_2\big)=f(a),
$$
with the strict inequality unless $\lambda\in\{0,1\}$ or $a_1=a_2$. It follows that the minimum is attained only when~$q(z)=q_\sigma(z):=(\sigma+z)/(\sigma-z)$, where $\sigma\in\UC$ is uniquely determined by $q_\sigma(\tau)=ia$, i.e. $\sigma=-\tau(1+ia)/(1-ia)$.

The function~$q_\sigma$ is the unique extreme point of~$\Cara^*(\tau,a)$ at which the minimum of $q\zrw(\tau)$ is attained. To complete the proof, it remains to show that there are no other (non-extreme) points~$q\in\Cara^*(\tau,a)$ of minimum.
Indeed, consider the preimage of the minimal value, i.e. the set $E:=\{q\in\Cara^*(\tau,a)\colon q\zrw(\tau)=q_\sigma\zrw(\tau)\}=\{q\in\Cara^*(\tau,a)\colon q\zrw(\tau)\le q_\sigma\zrw(\tau)\}$. Clearly, $E$ is a face for~$\Cara^*(\tau,a)$. In particular, it is a convex set. Moreover, by Theorem~\ref{TH_regular-contact-point-of-HFunc}\,(B), $E$ is closed. Since $E\subset\Cara$ and $\Cara$ is compact, it follows that $E$ is also compact. Therefore, the standard argument applies even though $\Cara^*(\tau,a)$ is not compact itself. Namely, if $E\neq\{q_\sigma\}$, then by Krein\,--\,Milman Theorem~\ref{TH_KreinMilman} we would have that $E$ has at least two distinct extreme points, which are in turn extreme points of~$\Cara^*(\tau,a)$. This would constitute a contradiction with the uniqueness of the minimum on $\extr\Cara^*(\tau,a)$, so the proof is complete.
\end{proof}

\begin{remark}
The set of Carath\'eodory functions~$\Cara(\tau,a)$ considered in Proposition~\ref{PR_in-Cara-class} is not compact. However, arguing as in the proof of Theorem~\ref{TH_regular-contact-point-of-HFunc}(B), one can show that $q\zrw(\tau)$ tends to its minimal value along a sequence~$(q_n)\subset\Cara(\tau,a)$ if and only if only if $(q_n)$~converges to the extremal function~$q_\sigma$.
\end{remark}

\begin{theorem}\label{TH_value-region-boundary}
Let $\tau\in\UC\setminus F$. In the above notation, 
$$
 V_\tau(F,\Lambda):=\Big\{\big(G(0),\lambda(G)\big)\in\C\times\Real\colon G\in\GenDW(F,\Lambda)\Big\}=
 \Big\{(\zeta,\omega)\colon\zeta\in Z,\,\omega\in I_\zeta\Big\},
$$
where:
\begin{itemize}
\item[(i)] for any $\zeta\in\interior Z$, $I_\zeta$ is the interval $\big[0,\omega_\zeta\big]$, $\omega_\zeta:=f\big(\ell_\zeta\,+\,i\sum_{k=1}^n|\lambda_k|^{-1}\Im(\overline\sigma_k\tau)\big)$,
$$
 f(w):=\frac{2\Re w}{|w|^2\,+\,2\Re w\,\sum_{k=1}^n|\lambda_k|^{-1}},
$$
with $\ell_\zeta$ and $Z$ defined as in Theorem~\ref{TH_value-region-interior};
\vskip.75ex

\item[(ii)]
for any $\zeta\in\partial Z$, $I_\zeta$ is a singleton: namely, if $1/\overline{\zeta}=\sum_{k=1}^n|\lambda_k|^{-1}(\tau-\sigma_k)$, then $I_\zeta=\big\{\big(\sum_{k=1}^n|\lambda_k|^{-1}\big)^{-1}\big\}$; otherwise, $I_\zeta=\{0\}$.
\end{itemize}
\vskip.75ex

Moreover, for each $\zeta\in\interior Z$, there exists a unique $G\in\GenDW(F,\Lambda)$ such that $G(0)=\zeta$ and $\lambda(G)=\omega_\zeta$; it is given by
\begin{equation}\label{EQ_extr-func-boundary}
G(z)=\widetilde G_\zeta(z):=\frac{(\tau-z)(1-\overline\tau z)}%
{\displaystyle \frac{\sigma_\zeta+z}{\sigma_\zeta-z}\Re \ell_\zeta+i\Im \ell_\zeta+ p_0(z;F,\Lambda)},
\end{equation}
where $\sigma_\zeta\in\UC$ is uniquely defined by the condition that the denominator in~\eqref{EQ_extr-func-boundary} vanishes at~$z=\tau$.
\end{theorem}

\begin{proof}
Following the proof of Theorem~\ref{TH_value-region-interior}, we see that also for ${\tau\in\UC}$,
the range of $\GenDW(F,\Lambda)\ni G\mapsto G(0)$ coincides with~$Z$ and that for each $\zeta\in\partial Z$, there exists a unique $G=G_\zeta\in\GenDW(F,\Lambda)$ such that $G(0)=\zeta$. Namely, $G_\zeta$ is given by~\eqref{EQ_G-zeta} if $\zeta\neq0$, and $G_0\equiv0$. Therefore, part~(ii)  of Theorem \ref{TH_value-region-boundary} can be verified by a simple computation.

Now fix $\zeta\in\interior Z$ and let us find the range~$I_\zeta\subset\Real$ of $G\mapsto \lambda(G)$ on the set $\mathcal G_\zeta:=\big\{G\in\GenDW(F,\Lambda)\colon G(0)=\zeta\big\}$. As in the proof of Theorem~\ref{TH_value-region-interior}, we can write
\begin{equation}\label{EQ_GenDW-q}
G(z)=\frac{(\tau-z)(1-\overline \tau z)}{\displaystyle q(z)\Re\ell_\zeta+i\Im\ell_\zeta+p_0(z;F,\Lambda)},\quad z\in\UD,
\end{equation}
where $q\in\Cara$. Trivially, $\lambda(G)\ge0$, and moreover, $\lambda(G)=0$ if we set $q\equiv 1$ in the above representation~\eqref{EQ_GenDW-q}. Hence, $\min I_\zeta=0$. Note also that $\mathcal G_\zeta$ is convex and that for a fixed~$\tau$, $G\mapsto \lambda(G)$ is linear. Therefore, $I_\zeta$ is an interval and it remains to find~$\max I_\zeta$.

Denote by $p_*$ the denominator of~\eqref{EQ_GenDW-q}. Suppose first that $\lambda(G)\neq0$. Then there exists finite angular limit
\begin{equation}\label{EQ_lim-explan}
\frac{1}{\lambda(G)}=\anglim_{z\to\tau}\frac{z-\tau}{G(z)}= \anglim_{z\to\tau}\frac{p_*(z)}{1-\overline\tau z}
\end{equation}
and hence $p_*$ has a regular null-point at~$\tau$. Note that $p_0(\cdot;F,\Lambda)$ is holomorphic at~$\tau$ and $\Re p_0(\tau;F,\Lambda)=0$. It follows that $q$ has a regular contact point at~$\tau$ and we easily see that $$\lambda(G)=1/p\zrw_*(\tau)=\big(q\zrw (\tau)\Re\ell_\zeta+p\zrw_0(\tau;F,\Lambda)\big)^{-1}.$$

This formula holds true under the weaker assumption that $\anglim_{z\to\tau} p_*(z)=0$. Indeed, in this case, $\tau$ is still a contact point of~$q$. If it is not regular, then by the above argument $\lambda(G)=0$, and $q\zrw(\tau)=+\infty$ by the very definition.

Finally, notice that if $\anglim_{z\to\tau}p_*(z)$ does not exists or it is different from~$0$, then the limit in~\eqref{EQ_lim-explan} cannot be finite. Hence, in such a case, $\lambda(G)=0$.


Thus, finding $\max I_\zeta$ reduces to the problem solved in Proposition~\ref{PR_in-Cara-class} with
$$
a:=-\frac{\Im\ell_\zeta+\Im p_0(\tau;F,\Lambda)}{\Re\ell_\zeta}=-\frac{\textstyle \Im\ell_\zeta+\sum_{k=1}^n|\lambda_k|^{-1}\Im(\overline\sigma_k\tau)}
{\Re\ell_\zeta},
$$
and it is just a computation to check the expressions for $\max I_\zeta$ and for the unique extremal function given in the statement of Theorem \ref{TH_value-region-boundary}.
\end{proof}

Theorems~\ref{TH_value-region-interior}\,--\,\ref{TH_value-region-boundary} imply the following well-known result (see \cite{ElShTa11}, \cite{CD}).
\begin{corollary}\label{CR_lambda}
The following assertions hold.
\begin{itemize}
\item[(A)] If $\tau\in\UD$, then the range of  $G\mapsto\lambda(G)$ on $\GenDW(F,\Lambda)$ is $$\{\omega\colon|\omega-r|\le r\},\quad\text{where $~r:=\textstyle\big(\sum_{k=1}^n|\lambda_k|^{-1}\big)^{-1}$.}$$ Each boundary point is delivered by exactly one function $G\in\GenDW(F,\Lambda)$, and the family of all such functions coincides with~$\{G_\zeta\colon\zeta\in\partial Z\}$, where $G_\zeta$'s are defined in Theorem~\ref{TH_value-region-interior}.

    \item[(A${}'$)] In particular, for $\tau\in\UD$ we have the sharp estimate
$\Re\lambda(G)\le 2\,\big(\sum_{k=1}^n|\lambda_k|^{-1}\big)^{-1}$ for all $G\in \GenDW(F,\Lambda)$, with the equality only for
$$
 G(z)=\frac{(\tau-z)(1-\overline\tau z)}{p_0(z;F,\Lambda)-i\sum_{k=1}^n|\lambda_k|^{-1}\Im(\overline\sigma_k\tau)},\quad z\in\UD.
$$

\item[(B)] If $\tau\in\UC\setminus F$, then the sharp estimate $\lambda(G)\le \big(\sum_{k=1}^n|\lambda_k|^{-1}\big)^{-1}$ holds for any ${G\in\GenDW(F,\Lambda)}$, with the equality occurring  only for the function $G$ defined by the same formula as in~(A${}'$).
\end{itemize}
\end{corollary}

\begin{proof}
Fix $\tau\in\UD\setminus\{0\}$ and apply Theorem~\ref{TH_value-region-interior}. Instead of $\lambda(G)$, we will consider the quantity $\eta(G):=(1-|\tau|^2)/\lambda(G)$. Note that the range of $\zeta\mapsto\ell_\zeta$ on~$Z\setminus\{0\}$ is the closed half-plane $\{w\colon \Re w\ge0\}$. Moreover, for each fixed $C\ge0$, the union of disks $\{(1-|\tau|^2)/\omega\colon\omega\in\Omega_\zeta\}$ over all $\zeta\in Z$ with $\Re\ell_\zeta=C$ is equal to
$$
\Big\{\eta\colon \frac{1-|\tau|^2}{1+|\tau|^2}C \le\Re \big(\eta-p_0(\tau;F,\Lambda)\big)\le\frac{1+|\tau|^2}{1-|\tau|^2}C\Big\}.
$$
Taking the union over all $C\ge0$, we see that the range of $G\mapsto\eta(G)$ on ${\GenDW(F,\Lambda)\setminus\{G\equiv0\}}$ is the closed half-plane described by~$\Re\eta\ge\Re P_0(\tau;F,\Lambda)=\tfrac12(1-|\tau|^2)\sum_{k=1}^n|\lambda_k|^{-1}$, with the boundary corresponding to $\Re\ell_\zeta=0$, i.e. to the infinitesimal generators $G_\zeta$, $\zeta\in\partial Z\setminus\{0\}$, defined by~\eqref{EQ_G-zeta}.

This proves (A) and (A${}'$) for $\tau\neq0$; and for~$\tau=0$, these two assertions follow immediately from Theorem~\ref{TH_value-region-tau-zero}.

Now let us assume $\tau\in\UC\setminus F$. In this case, we have to apply Theorem~\ref{TH_value-region-boundary}.
Note that $\Re\ell_\zeta>0$ for all $\zeta\in \interior Z$. The function $f$ defined in Theorem~\ref{TH_value-region-boundary} satisfies $$f(u+iv)<f(u)<\lim_{\varepsilon\to 0^+}f(\varepsilon)=\big(\sum_{k=1}^n|\lambda_k|^{-1}\big)^{-1}$$ for all $u>0$ and $v\in\Real$. Therefore, according to Theorem~\ref{TH_value-region-boundary}, the maximum of $\lambda(G)$ on $\GenDW(F,\Lambda)$ coincides with the r.h.s. of the above inequality and it is attained only of $G=G_\zeta$ with $\zeta:=\overline{\big(\sum_{k=1}^n|\lambda_k|^{-1}(\tau-\sigma_k)\big)^{-1}}$. This proves~(B).
\end{proof}

If $\tau\in\UD$, then {$\Re \lambda(G)=0$} may happen only for a very narrow class of one-parameter semigroups: all elements of such semigroups are Moebius transformations of~$\UD$. However, if $\tau\in\UC$, then $\lambda(G)=0$ means simply that the one-parameter semigroup is parabolic, which is probably the most interesting and complicated case. Theorem~\ref{TH_value-region-boundary2} below deals with one important class of parabolic semigroups.

\begin{remark}\label{RM_beta}
Let $\tau\in\UC$ and let $G$ be the infinitesimal generator of a non-trivial one-parameter semigroup with the DW-point~at~$\tau$.
Then, by the Berkson\,--\,Porta formula~\eqref{EQ_BP-formula},
$G(z):=\tau(1-\overline\tau z)^2p(z)$, where $p\not\equiv0$ is a Herglotz function. Therefore, the limit
\begin{equation}\label{EQ_fla-for-betta}
\beta(G):=\anglim_{z\to\tau}\frac{(\tau-z)^3}{\tau^2G(z)}=\big(1/p\big)\plw(\tau)
\end{equation}
exists and belongs to~$[0,+\infty)$.
\end{remark}
It {can be seen} 
that $\beta(G)\neq0$ if and only if $G$ has angular derivatives at~$\tau$ up to the third order and $G'(\tau)=G''(\tau)=0$. The Cayley map $\UD\ni z\mapsto (\tau+z)/(\tau-z)\in\UH$ establishes a one-to-one correspondence between  one-parameter semigroups in~$\UD$ with the DW-point at~$\tau$ and one-parameter semigroups in~$\UH$ with the DW-point at~$\infty$. Under this correspondence, infinitesimal generators $G\not\equiv0$  in~$\UD$ with $\beta(G)\neq0$ are transformed to infinitesimal generators in~$\UH$ that can be characterized as holomorphic functions ${H:\UH\to\UH}$ with the asymptotic expansion $H(\zeta)={\ell(H)/\zeta+\gamma(\zeta)}$, where $\ell(H)=4/\beta(G)\in(0,+\infty)$ and $\anglim_{\zeta\to\infty}\zeta\gamma(\zeta)=0$. Such functions~$H$ play an important role in the chordal Loewner Theory, see e.g.~\cite{Goryainov-Ba, Bauer}.

The following theorem gives the sharp estimate of~$\beta(G)$ for $G\in\GenDW(F,\Lambda)$ with a prescribed value~$G(0)$.

\begin{theorem}\label{TH_value-region-boundary2}
Let $\tau\in\UC\setminus F$. The value region
$$
W_\tau(F,\Lambda):=\Big\{\big(G(0),\beta(G)\big)\in\C\times\R\colon G\in\GenDW(F,\Lambda)\setminus\{G\equiv0\}\Big\},
$$
coincides with the set $\Big\{(\zeta,b)\colon\zeta\in Z\setminus\{0\},\,0\le b\le 2\Re\ell_\zeta\Big\},$
where $Z$ and $\ell_\zeta$ are defined as in the Theorem~\ref{TH_value-region-interior}.

Moreover, for each $\zeta\in Z\setminus\{0\}$, there exists a unique $G\in\GenDW(F,\Lambda)$ such that $G(0)=\zeta$ and $\beta(G)=2\Re\ell_\zeta$; it is given by
\begin{equation}\label{EQ_G-zeta,tau}
G(z)=G_{\zeta,\tau}(z):=\frac{(\tau-z)(1-\overline\tau z)}%
{\displaystyle \frac{\tau+z}{\tau-z}\Re \ell_\zeta+i\Im \ell_\zeta+ p_0(z;F,\Lambda)},\quad\text{for all $z\in\UD$}.
\end{equation}
\end{theorem}

\begin{proof}
Following the proof of Theorem~\ref{TH_value-region-interior}, we see that the range of the functional $G\mapsto G(0)$ on~$\GenDW(F,\Lambda)\setminus\{G\equiv0\}$ coincides with $Z\setminus\{0\}$ and moreover, if $\zeta\in\partial Z\setminus\{0\}$, then there exists exactly one $G\in\GenDW(F,\Lambda)\setminus\{G\equiv0\}$ with $G(0)=\zeta$, namely, $G=G_\zeta$, see~\eqref{EQ_G-zeta}. Therefore, in case $G(0)\in\partial Z\setminus\{0\}$ we simply have~$\beta(G)=0$.

Suppose now that $G(0)=:\zeta\in\interior Z$. Again, as in the proof of Theorem~\ref{TH_value-region-interior}, we can write
$$
G(z)=\frac{(z-\tau)(1-\overline\tau z)}{p(z)+p_0(z;F,\Lambda)},\quad p(z):=q(z)\Re\ell_\zeta+i\Im\ell_\zeta,
$$
where $q$ is an arbitrary function from the Carath\'eondory class~$\Cara$. Therefore, the problem to find the range of~$G\mapsto \beta(G)$ among all $G\in\GenDW(F,\Lambda)\setminus\{G\equiv0\}$ with ${G(0)=\zeta}$ is now reduced to finding the maximum of $q\plw(\tau)$. The linear functional $q\mapsto q\plw(\tau)$ satisfies the hypothesis of Theorem~\ref{TH_max-on-extreme} with~$X:=\Hol(\UD)$ and $K:=\Cara$. Taking into account Remark~\ref{RM_Caratheodory-class}, we conclude that it is sufficient to consider the functions of the form~$q_\sigma(z):={(\sigma+z)/(\sigma-z)}$, where ${\sigma\in\UD}$. For any $\sigma\neq\tau$, we have $q_\sigma\plw(\tau)=0$, which is clearly the minimal value of the functional on~$\Cara$, while for $\sigma=\tau$, it attains its maximal value ${q_\tau\plw(\tau)=2}$. Therefore, the range of $q\mapsto q\plw(\tau)$ over~$\Cara$ is~$[0,2]$, which immediately implies the assertion concerning the value region~$W_\tau(F,\Lambda)$.

Moreover, among the extreme points of~$\Cara$, there exists only one function~$q=q_\tau$ for which  $q\plw(\tau)$ attains its maximal value. It follows that the maximum over the whole class~$\Cara$ is attained only for~$q_\tau$, see again Theorem~\ref{TH_max-on-extreme}, from which we immediately obtain the remaining part of the theorem.
\end{proof}


\subsection{Extreme points of $\GenDW(F,\Lambda)$}
Another method to obtain results similar to theorems given in Sect.\,\ref{SS_valueReg-ellipt} and~\ref{SS_valueReg-boundary} is based on looking for the extreme points of~$\GenDW(F,\Lambda)$.
\begin{theorem}\label{TH_extereme-points-bis}
Let $\tau\in\overline\UD\setminus F$. Every extreme point $G\not\equiv0$ of the class $\GenDW(F,\Lambda)$ is of the form
\begin{equation}\label{EQ_extreme-point-GenDW}
G(z)=\frac{(\tau-z)(1-\overline\tau z)}{\displaystyle ib\,+\,\sum_{j=1}^{n-1} a_j\frac{s_j+z}{s_j-z}\,+\,p_0(z;F,\Lambda)},\quad z\in\UD,
\end{equation}
where $s_1,\ldots,s_{n-1}\in\UC$, $a_1,\ldots ,a_{n-1}\ge0$, $b\in\Real$, and $p_0$ is defined by~\eqref{EQ_p0-alpha_k}.
Some of the points $s_j$'s may belong to~$F$.
\end{theorem}
In the proofs we will need the following lemma. Fix some~$m\in\mathbb{N}$ and consider the set $\Herg_m$ of all Herglotz functions of the form
\begin{equation}\label{EQ_Herg-m}
p(z)=ib\,+\,\sum_{j=1}^m a_j\frac{s_j+z}{s_j-z},\quad z\in\UD,
\end{equation}
where $b\in\Real$, $a_1,\ldots, a_m>0$ and $s_j,\ldots,s_m$ are pairwise distinct points on~$\UC$.
\begin{lemma}\label{LM_involution}
The map $p\mapsto 1/p$ is an involution of~$\Herg_m$ onto itself.
\end{lemma}
\begin{proof}
Clearly, every function $p$ of the form~\eqref{EQ_Herg-m} is a rational function of degree~$m$ with all poles being simple and lying on~$\UC$. Moreover, ${\Re p(z)>0}$ for all ${z\in\UD}$, and ${\Re p(z)<0}$ for all~$z\in\ComplexE\setminus\overline\UD$. It follows that all zeros of~$p$ are simple and belong to~$\UC$. Denote them by $\kappa_1,\ldots,\kappa_m$.

For each $j=1,\ldots,m$, $p\zrw(\kappa_j)=-\kappa_jp'(\kappa_j)\in (0,+\infty)$ because $p$ is a non-trivial Herglotz function and it is holomorphic at~$\kappa_j$; see Remark~\ref{RM_p-pole-zero}. The rational function
$$
R(z):=\frac{1}{p(z)}\,-\,\sum_{j=1}^m\frac{1}{2p\zrw(\kappa_j)}\frac{\kappa_j+z}{\kappa_j-z}
$$
has no poles in~$\ComplexE$, and on $i\mathbb R\setminus\{\kappa_1,\ldots,\kappa_n\}$ its real part vanishes. Therefore, $R$ is an imaginary constant.

This shows that $1/p\in\Herg_m$ for any~$p\in\Herg_m$. The proof is complete.
\end{proof}

\begin{proof}[{\fontseries{bx}\selectfont Proof of Theorem \ref{TH_extereme-points-bis}}]
Let $G\in\GenDW$. By the Berkson\,--\,Porta and the Riesz\,--\,Herglotz representation representation formulas~\eqref{EQ_BP-formula},\,\eqref{EQ_HerglotzRepresentation},
\begin{equation}\label{EQ_BP-RszHgtz}
 G(z)=(\tau-z)(1-\overline\tau z)\Big(\alpha \int_{\UC}\frac{\sigma+z}{\sigma-z}\,\di\mu(\sigma)~+~i\beta\Big), \quad z\in\UD,
\end{equation}
with some $\alpha\ge0$, $\beta\in\Real$, and some Borel probability measure $\mu$ on~$\UC$.

Moreover, by Theorem~\ref{TH_BRFP-withPommerenke} and Lemmas~\ref{LM_Sarason} and \ref{LM_contact-point}, $G\in\GenDW(F,\Lambda)\setminus\{G\equiv0\}$ if and only if $\alpha>0$ and the measure~$\mu$ in~\eqref{EQ_BP-RszHgtz} satisfies
\begin{gather}\label{EQ_measure-cond1}
 \int_{\UC}|\varsigma-\sigma_k|^{-2}\,\di\mu(\varsigma)\,\le\,\alpha_k/(2\alpha),\quad k=1,\ldots, n,\\
 \intertext{where $\alpha_k$'s are defined in~\eqref{EQ_p0-alpha_k}, and}\label{EQ_measure-cond2}
 \int_0^{2\pi}\cot(\theta/2)\di\mu(\sigma_k e^{i\theta})\,+ \,\beta/\alpha\,=\,0, \quad k=1,\ldots, n.
\end{gather}

If $G\not\equiv0$ is an extreme point for~$\GenDW(F,\Lambda)$, then $\mu$ is an extreme point for the set of all Borel probability measures on~$\UC$ subject to conditions~\eqref{EQ_measure-cond1} and~\eqref{EQ_measure-cond2}.

Therefore, on the one hand, by Theorem~\ref{TH_W-extr},
$\mu$ is a linear combination of finite (in fact, at most $2n+1$) Dirac measures on~$\UC$,
and hence $G(z)=(\tau-z)(1-\overline\tau z)p(z)$, where $p\in\Herg_m$ for some~$m\in\Natural$, i.e.
$$
p(z)=i\beta\,+\,\sum_{j=1}^{m}v_jq_j(z),\quad q_j(z):=\frac{\kappa_j+z}{\kappa_j-z},\quad z\in\UD,
$$
with $v_1,\ldots,v_m>0$ and pairwise distinct $\kappa_1,\ldots,\kappa_m\in\UC$.

On the other hand, $G\in\GenDW(F,\Lambda)$ and hence, by Theorem~\ref{TH_representation-formula}\,(A), $1/p(z)=\tilde p(z)+p_0(z;F,\Lambda)$, $z\in\UD$, for some Herglotz function~$\tilde p$.

By Lemma~\ref{LM_involution}, $1/p\in\Herg_m$. Therefore, $m\ge n$ and $\tilde p$ is of the form
$$
\tilde p(z)=ib\,+\,\sum_{j=1}^{m'}a_j\frac{s_j+z}{s_j-z},\quad z\in\UD,
$$
where $m-n \le m'\le m$, $a_1,\ldots,a_{m'}>0$, and $s_1,\ldots,s_{m'}$ are pairwise distinct points on~$\UC$, of which exactly $\nu:=m'-(m-n)$ belong to~$F$.

It remains to show that $m'\le n-1$. To this end we notice that
$$
 G_t(z):=(\tau-z)(1-\overline\tau z)\Big(p(z)+t\Big(
ix_0\,+\,\sum_{j=1}^{m}x_jq_j(z)\Big)\Big),\quad z\in\UD,
$$
belongs to $\GenDW(F,\Lambda)$ for all $t\in\Real$ small enough provided that $(x_0,x_1,\ldots,x_m)\in\Real^{m+1}$ solves the linear system
\begin{align*}
x_0+\displaystyle \frac{1}{i}\sum_{j=1}^m q_j(\sigma_k)x_j&=0, & k=1,\ldots, n,\\[.75ex]
\frac{1}{\sigma_k}\sum_{j=1}^m q_j'(\sigma_k)x_j&=0, & k\in J,
\end{align*}
where $J$ consists of integers $k=1,\ldots,n$ such that $\sigma_k\not\in\{s_1,\ldots, s_{m'}\}$.

Taking into account that all the coefficients in the above homogeneous system are real, we see that it cannot have non-trivial solutions, because otherwise $G$ would not be an extreme point of~$\GenDW(F,\Lambda)$. It follows that the number of unknowns, which is equal to~$m+1$, cannot exceed the number of equations, which is~$2n-\nu=n+m-m'$.

Thus, ${m'\le n-1}$. To complete the proof, we mention that the case $m'<n-1$ is of course possible. Therefore, some  coefficients $a_j$ in representation~\eqref{EQ_extreme-point-GenDW} may vanish.
\end{proof}

Thanks to general results in the Krein\,--\,Milman Theory, see e.g. \cite[\S18.1.3]{Kadec}, Theorem~\ref{TH_extereme-points-bis} implies that $G\in\GenDW(F,\Lambda)$ if and only if it admits a representation in terms of a regular Borel measure~$\mu$ supported on the finite-dimensional family in~$\GenDW(F,\Lambda)$ defined by formula~\eqref{EQ_extreme-point-GenDW} and having total weight~$|\mu|\le1$. (It is possible that $|\mu|<1$,  because $G\equiv0$ is an extreme point of $\GenDW(F,\Lambda)$, but it does not belong to the aforementioned family.)
In particular, for~$n=1$ we recover a result of Goryainov and Kudryavtseva.
\begin{corollary}[\protect{\cite[Theorem~1]{Goryainov-Kudryavtseva}}]\label{CR_GorKudr}
Let $\lambda<0$, $\sigma\in\UC$, $\tau\in\overline{\UD}\setminus\{\sigma\}$. Then $G\in\GenDW(\sigma,\lambda)$ if and only if there exists a Borel probability measure $\mu$ on~$\UC$ such that
\begin{equation}\label{EQ_GK-representation}
G(z)=\frac{|\lambda|}{|\sigma-\tau|^2}(\tau-z)(1-\overline\tau z)(1-\overline{\sigma}z)\int_{\UC}\frac{1-\kappa}{1-\kappa \overline{\sigma}z}\di\mu(\kappa).\quad\text{for all~$z\in\UD$}.
\end{equation}
\end{corollary}

\begin{proof}
By Theorem~\ref{TH_extereme-points-bis} for the case $n=1$, $F:=\{\sigma\}$, $\Lambda:=\{\lambda\}$, the extreme points of~$\GenDW(\sigma,\lambda)$ other than identical zero are of the form $$G_b(z)={(\tau-z)(1-\overline\tau z)\big(ib+p_0(z;\sigma,\lambda)\big)^{-1}},$$ where $b\in\Real$. According to the Krein\,--\,Milman Theorem in integral form, see e.g. \cite[\S18.3.1]{Kadec}, ${G\in\GenDW(\sigma,\lambda)}$ if and only if it can be represented as the integral of $\Real\cup\{\infty\}\ni b\mapsto G_{b}$, where $G_\infty(z):=0$ for all ${z\in\UD}$, against some Borel probability measure on~$\Real\cup\{\infty\}$.

To simplify expressions, we introduce a new parameter $\kappa:=(iy-1)/(iy+1)\in\UC$, where $y:=2b|\lambda|/|\sigma-\tau|^2$. Then
\begin{equation*}
G_{b(\kappa)}(z)=\frac{|\lambda|}{|\sigma-\tau|^2}(\tau-z)(1-\overline\tau z)(1-\overline{\sigma}z)\frac{1-\kappa}{1-\kappa \overline{\sigma}z},\quad
z\in\UD,
\end{equation*}
with $\kappa=1$ corresponding to~$G_\infty$. This immediately leads to~\eqref{EQ_GK-representation}.
\end{proof}

The next theorem refers to another family of infinitesimal generators with given boundary regular null-points. Let $F=\{\sigma_k\}_{k=1}^n$ be as above and $\tau\in\overline\UD\setminus F$. Denote by $\GenDW(F)$ the class of all infinitesimal generators $G\in\GenDW$ such that the corresponding one-parameter semigroup $(\phi_t^G)$ has BRFPs at $\sigma_k$'s with repelling spectral values $\lambda_k$ satisfying ${\sum_{k=1}^n|\lambda_k|\le1}$. As usual, we regard $G\equiv0$ to be an element of~$\GenDW(F)$.

\begin{remark}\label{RM_compact}
Arguing as in the proof of Corollary~\ref{CR_compact}, one can easily show that $\GenDW(F)$ is a compact convex set in~$\Hol(\UD)$ and that the class $\GenDWe(F)$ formed by all $G\in\GenDW(F)$ for which the equality $\sum_{k=1}^n|\lambda_k|=1$ holds is a convex dense subset of~$\GenDW(F)$.
\end{remark}

\begin{theorem}\label{TH_extereme-points-tre}
In the above notation, $G\not\equiv0$ is an extreme point of $\GenDW(F)$ if and only if it is of the form
\begin{equation}\label{EQ_extereme-points-tre}
G(z)=\frac{(\tau-z)(1-\overline\tau z)}{ib+p_0(z;F,L)},\quad z\in\UD,
\end{equation}
where $b\in\Real$ and $L:=\{\lambda_k\}_{k=1}^n\in(-\infty,0)^n$ satisfies $\sum_{k=1}^n|\lambda_k|=1$.
\end{theorem}

\begin{proof}
The fact that every extreme point of $\GenDW(F)$ different from the identical zero is of the form~\eqref{EQ_extereme-points-tre} can be established using essentially the same arguments as in the proof of Theorem~\ref{TH_extereme-points-bis}. By this reason, we omit the details.

To prove the converse, consider a function $G$ of the form~\eqref{EQ_extereme-points-tre}. We have to show that $G$ is an extreme point of~$\GenDW(F)$. Clearly, by Theorem~\ref{TH_representation-formula}\,(A), $G\in\GenDW(F)$.

Moreover, by Lemma~\ref{LM_involution}, $G(z)=(\tau-z)(1-\overline\tau z)p(z)$, where $p\in\Herg_n$ with $p(\sigma_k)=0$ for $k=1,\ldots,n$.
If $G$ is not an extreme point of $G\in\GenDW(F)$, then there exist Herglotz functions $p_1\neq p_2$ such that  $p=(p_1+p_2)/2$, with the infinitesimal generators $G_j(z):=(\tau-z)(1-\overline\tau z)p_j(z)$, $j=1,2$, belonging to~$\GenDW(F)$.

Denote $$\mathcal L(G):=\sum_{k=1}^n|\lambda_k(G)|,$$ where $\lambda_k(G)=-G'(\sigma_k)<0$ is  the spectral value of~$(\phi^G_t)$ at~$\sigma_k$. We have $2=2\mathcal L(G)={\mathcal L(G_1)+\mathcal L(G_2)}$, and $\mathcal L(G_j)\le1$ because $G_j\in\GenDW(F)$, $j=1,2$. It follows that $\mathcal L(G_1)=\mathcal L(G_2)=1$. In particular, $p_1$ and $p_2$ are non-trivial Herglotz functions.

On the one hand, since the Herglotz measure of~$p$ is a linear combination of $n$ Dirac measures on~$\UC$, the same holds for $p_1$ and $p_2$, i.e. $p_1,p_2\in\bigcup_{m=1}^n\Herg_m$ and every pole of $p_1$ or $p_2$ is also a pole of~$p$.

On the other hand, $p_j(\sigma_k)=0$, $j=1,2$, $k=1,\ldots,n$. In particular, $p_1$, $p_2$ are rational functions of degree at least~$n$. It follows that $p_1,p_2\in\Herg_n$.

Therefore, the rational functions $p$, $p_1$, $p_2$ have exactly the same zeros and poles, all of them simple. Moreover, at all other points on~$\UC$, these functions take purely imaginary values. Finally, their reals parts are positive in~$\UD$. It follows that $p_j=\gamma_jp$, $j=1,2$,  with some $\gamma_1,\gamma_2>0$. From $\mathcal L(G)=\mathcal L(G_1)=\mathcal L(G_2)=1$, we conclude that $\gamma_1=\gamma_2=1$ and hence $p=p_1=p_2$. This means that $G$ is indeed an extreme point of~$\GenDW(F)$.
\end{proof}

\section{Loewner-Kufarev-type ODE for self-maps with BRFPs}\label{S_param-and-CP-ineq}
In this section we combine our results with the theory developed
in~\cite{BRFPLoewTheory,Gumenyuk_parametric,Gumenyuk_parametric2} in order to develop a parametric representation of
univalent self-maps $\varphi\in\Hol(\UD,\UD)$ with given boundary regular fixed points based on a Loewner\,--\,Kufarev-type
ODE. Note that in this case, in contrast to the previous sections, we do not suppose that~$\varphi$ is an element of a
one-parameter semigroup. As an application of this parametric representation, we will give a new proof of the
Cowen\,--\,Pommerenke inequalities for univalent self-maps of the unit disk.

Denote by $\U$ the class of all univalent holomorphic mappings ${\varphi:\UD\to\UD}$ and let $\Ut\tau$, $\tau\in\overline\UD$,
be the subclass of~$\U$ formed by $\id_\UD$ and all ${\varphi\in\U\setminus\{\id_\UD\}}$ whose Denjoy\,--\,Wolff point coincides
with~$\tau$.
Furthermore, given a finite set $F\subset\UC$, consider the class $\UF[F]$ of all~${\varphi\in\U}$ satisfying the following
condition: every $\sigma\in F$ is a boundary regular fixed point of~$\varphi$. Let $\Ut\tau[F]:=\Ut\tau\cap\UF[F]$ for any
$\tau\in\overline\UD\setminus F$.

\subsection{Parametric representation}\label{S_param-representation}
This section is devoted to the proof of the following result. It makes use of an intrinsic version of Loewner Theory in the unit disk developed in~\cite{BCM1}. We refer the reader to that paper for the terminology and basic results.

\begin{theorem}\label{TH_param-repres}Let $\tau\in\overline\UD$ and $F:=\{\sigma_1,\ldots,\sigma_n\}\subset\UC\setminus\{\tau\}$, where $\sigma_k$'s are pairwise distinct. Then the following statements hold.\medskip\\
    \textrm{\bf (A)}  For any
    $\varphi\in\Ut\tau[F]\setminus\id_\UD$, there is a function $G:\UD\times[0,T]\to\Complex$, ${T:=\log\prod_{k=1}^n\varphi'(\sigma_k)}$, such that:
    \begin{itemize}
        \item[(i)]  for any~$z\in\UD$, $G(z,\cdot)$ is measurable on~$[0,T]$;
        \item[(ii)]  for a.e. $t\in[0,T]$, $G(\cdot,t)\in\GenDWe(F)$, {that is, for a.e. $t\in[0,T]$, $G(\cdot,t)$ is an infinitesimal generator such that the corresponding one-parameter semigroup has BRFPs at $\sigma_k$'s with repelling spectral values $\lambda_k$ satisfying ${\sum_{k=1}^n|\lambda_k|=1}$;}
        \item[(iii)] for any $z\in\UD$, $\varphi(z)=w_z(T)$, where $w=w_z(t)$ is the unique solution to
        \begin{equation}\label{EQ_param-repres}
        \frac{\di w}{\di t}=G\big(w(t),t\big),\quad t\in[0,T],\qquad w(0)=z.
        \end{equation}
    \end{itemize}
    \medskip
    \noindent{\bf (B)} Conversely, let $T>0$ and suppose that $G:\UD\times[0,T]\to\Complex$ satisfies~{\rm (i)} and~
\begin{itemize}
\item[(ii')] for a.e. $t\in[0,T]$, $G(\cdot,t)\in\GenDW(F),$  {that is, for a.e. $t\in[0,T]$, $G(\cdot,t)$ is an infinitesimal generator such that the corresponding one-parameter semigroup has BRFPs at $\sigma_k$'s with repelling spectral values $\lambda_k$ satisfying ${\sum_{k=1}^n|\lambda_k|\leq 1}$}.
\end{itemize}
Then for any~$z\in\UD$, the initial value problem~\eqref{EQ_param-repres} has a unique solution~$[0,T]\ni t\mapsto w=w_z(t)\in\UD$
and the maps~$\UD\ni z\mapsto \varphi_t(z):=w_z(t)$, $t\in[0,T]$, belong to~$\Ut\tau[F]$. Moreover, for each $k=1,\ldots,n$, $t\mapsto G'(\sigma_k,t)$ is an integrable function on~$[0,T]$ and
\begin{equation}\label{EQ-der}
    \log\varphi'_t(\sigma_k)=\int_0^{t}\!G'(\sigma_k,s)\,\di s\qquad\text{for
        all~$~t\in[0,T]$.}
    \end{equation}
\end{theorem}

\begin{proof}
Let $\varphi\in\Ut\tau[F]\setminus\id_\UD$. According to~\cite[Theorem~2]{Gumenyuk_parametric2} there exists an evolution family~$(\varphi_{s,t})\subset\Ut\tau[F]$ such that $\varphi=\varphi_{0,1}$.

Using \cite[Theorem~1.1]{BRFPLoewTheory}, we see that $f(t):=\log\prod_{k=1}^n\varphi'_{0,t}(\sigma_k)$ is locally absolutely continuous on~$[0,+\infty)$. Moreover, $\varphi'_{s,t}(\sigma_k)\ge1$ whenever $t\ge s\ge0$ and $k=1,\ldots,n$, with the equality possible only if $\varphi_{s,t}=\id_\UD$. Taking into account that $\varphi_{0,t}=\varphi_{s,t}\circ\varphi_{0,s}$ and using the Chain Rule for angular derivatives, see e.g. \cite[Lemma~2]{CoDiPo06}, we conclude that ${f(t)\ge f(s)}$ whenever $t\ge s\ge0$ and that the equality is only possible if $\varphi_{s,t}=\id_\UD$. It follows that there exists a family $(\psi_{s,t})_{T\ge t\ge s\ge0}\subset\Ut\tau[F]$, where $T:=f(1)=\log\prod_{k=1}^n\varphi'(\sigma_k)>0$, such that $\varphi_{s,t}=\psi_{f(s),f(t)}$ for any $s,t\in[0,1]$ with $t\ge s$. By construction,
$$
\log\prod_{k=1}^n\psi'_{0,t}(\sigma_k)=t\quad\text{for any $t\in[0,T]$}.
$$

We extend the family $(\psi_{s,t})$ to all $s\ge0$ and $t\ge s$ by setting $\psi_{s,t}:=\psi_{s,T}$ if $t\ge T\ge s\ge0$ and $\psi_{s,t}:=\id_\UD$ if $t\ge s\ge T$.
We claim that $(\psi_{s,t})$ is an evolution family. Indeed, consider one of the points in~$F$, e.g.~$\sigma_1$. For any $s\ge 0$ and any~$t\ge s$, we have
$$
 0\,\le\,\log \psi'_{s,t}(\sigma_1)\,\le~ \log\prod_{k=1}^n\psi'_{s,t}(\sigma_k)~=~\log\prod_{k=1}^n\psi'_{0,t}(\sigma_k)~-~ \log\prod_{k=1}^n\psi'_{0,s}(\sigma_k)~\le\, t-s
$$
and $\log \psi'_{0,t}(\sigma_1)-\log \psi'_{0,s}(\sigma_1)=\log \psi'_{s,t}(\sigma_1)$. Therefore, $t\mapsto \log \psi'_{0,t}(\sigma_1)$ is Lipschitz continuous on~$[0,+\infty)$. Note also that $\tau$ is the DW-point for all $\psi_{s,t}$'s different from~$\id_\UD$. According to \cite[Theorem~4.2]{Gumenyuk_parametric}, it follow that $(\psi_{s,t})$ is indeed an evolution family.

Let $G$ be the Herglotz vector field associated with~$(\psi_{s,t})$. Recall that $(\psi_{s,t})\subset\Ut\tau[F]$.
Hence according to \cite[Theorem~1.1]{BRFPLoewTheory} and \cite[Theorem~6.7]{BCM1}, for a.e. $s\ge0$, $G(\cdot,s)$ is the infinitesimal generator of a one-parameter semigroup contained in $\Ut\tau[F]$ and moreover,
$$
\log\psi_{0,t}'(\sigma_k)=\int_{0}^tG'(\sigma_k,s)\di s,\quad t\ge0,\quad k=1,\ldots,n.
$$
Since by construction, $\log\prod_{k=1}^n\psi'_{0,t}(\sigma_k)=t$ for all $t\in[0,T]$,  it follows that
$$
\sum_{k=1}^n G'(\sigma_k,s)=1\quad\text{for a.e.~${s\in[0,T]}$}.
$$
This shows that $G(\cdot,s)\in\GenDWe(F)$ for a.e. $s\in[0,T]$ and hence the proof of~(A) is complete.

To prove~(B), it is sufficient to show that if $G$ satisfies (i) and (ii'), then being extended by $G(\cdot,t)\equiv0$ for all~$t>T$ it becomes a Herglotz vector field, see \cite[Definitions~4.1 and~4.3]{BCM1}. In such a case, the conclusion in~(B) would follow from \cite[Theorems~4.4 and~5.2]{BCM1}, \cite[Corollary~7.2]{BCM1},  and \cite[Theorem~1.1]{BRFPLoewTheory} combined with Theorem~\ref{TH_BRFP-withPommerenke}.

In turn, to see that (i) and~(ii') imply that~$G$ is a Herglotz vector field, it is sufficient to recall, see Remark~\ref{RM_compact}, that~$\GenDW(F)$ is a compact class and hence for any compact set $K\subset\UD$ there exists $M(K)>0$ such that ${\max_{z\in K} |F(z)|\le M(K)}$ holds for all ${F\in\GenDW(F)}$. The proof is now complete.
\end{proof}

\subsection{Inequalities of Cowen and Pommerenke}\label{S_CP-ineq}
In this section we apply our results to give another proof of an inequality due to Cowen and Pommerenke.

To state rigorously the Cowen\,--\,Pommerenke inequality for univalent self-maps with the interior DW-point, we need the following lemma. Consider the class $\Ut\tau[F]$, where $\tau\in\UD$ and $F:={\{\sigma_1,\sigma_2,\ldots,\sigma_n\}} {\subset\UC}$ consists of $n$ distinct points.
\begin{lemma}\label{LM_log} Let  $\varphi\in \Ut\tau[F]$, $\tau\in\UD$.
\begin{itemize}
\item[(A)]There exists a single-valued branch $\Psi[\varphi]:\UD\to\Complex$ of $\log\big((\varphi(z)-\tau)/(z-\tau)\big)$ such that the angular limit of $\Psi[\varphi]$ vanishes at~$\sigma_k$ for each~$k=1,\ldots,n$.
\item[(B)] Moreover, let $(\varphi_{t})_{t\in I}$ be a family in~$\Ut\tau[F]$ over an interval $I\subset\Real$  such that $I\ni t\mapsto\varphi_t(z)$ is continuous for any ${z\in\UD}$. Suppose also that $t\mapsto \varphi'(\sigma_{k_0})$ is locally bounded on~$I$ for some $k_0=1,\ldots,n$. Then $I\ni t\mapsto \Psi[\varphi_t]\in\Hol(\UD,\C)$ is continuous (in the open-compact topology).
\end{itemize}
\end{lemma}
\vskip-.75ex\noindent The proof of this lemma is given in the Appendix. \vskip1ex

 Note that $\Psi[\varphi](\tau)$ is one of the values of $\log\varphi'(\tau)$, and this is the value that appears in the statement of the following theorem due to Cowen and Pommerenke.

\begin{theorem}[\protect{\cite[Theorem~7.1]{CowenPommerenke}}]\label{TH_C-P}
    Fix $\tau\in\UD$ {and an arbitrary finite sequence~$A:=(a_k)_{k=1}^n\subset(1,+\infty)$.} The value region
    $$U_\tau(F,A):=\{-\Psi[\varphi](\tau)\in\Complex\colon\varphi\in\Ut\tau[F],~\varphi'(\sigma_k)=a_k~\text{ for each}~k=1,\ldots,n\}$$
    is the closed disk $$\displaystyle
    D(A):=\big\{\omega\colon|\omega-r|\le r\big\},\quad r=r(A):=\Big(\sum\limits_{k=1}^n\dfrac1{\log a_k}\Big)^{-1},$$ with the
    point~${\omega=0}$ excluded. Each $\omega\in\partial U_\tau(F,A)\setminus\{0\}$ is delivered by a unique function~$\varphi_\omega$,
    which coincides with the element~$\phi^\omega_1$ of the one-parameter semigroup~$(\phi^\omega_t)$ associated with the
    infinitesimal generator
    $$G_\omega(z):=(\tau-z)(1-\overline\tau z)\left(
    i\gamma_\omega+\sum\limits_{k=1}^n\frac{|\tau-\sigma_k|^2}{2\log a_k}\,\frac{\sigma_k+z}{\sigma_k-z}\,\right)^{\!-1}\quad\text{for all~$z\in\UD$},$$
where $\gamma_\omega$ is a real constant depending on~$\omega$.
\end{theorem}
\begin{proof}
    Let $\varphi\in\Ut\tau[F]$, with $\varphi'(\sigma_k)=a_k$ for each ${k=1,\ldots,n}$. Then thanks to
    Theorem~\ref{TH_param-repres},
    \begin{align}
    \label{EQ_CP-rep1}
    \log\varphi'(\tau)&=\int_0^{T}\!G'(\tau,t)\di t,\quad T:=\log\prod_{k=1}^n a_k,\\
    \label{EQ_CP-rep2}
    \log a_k=\log\varphi'(\sigma_k)&=\int_0^T\! G'(\sigma_k,t)\di t,\quad k=1,\ldots,n,
    \end{align}
where $G$ is measurable in~$t\in[0,T]$ and $G(\cdot,t)\in\GenDWe(F)$ for a.e.~$t\in[0,T]$.
Note that the value of~$\log\varphi'(\tau)$ given by~\eqref{EQ_CP-rep1} coincides with~$\Psi[\varphi](\tau)$, defined in Lemma~\ref{LM_log}\,(A), thanks to part~(B) of the same lemma.

Representing the disk $\{\omega:|\omega-r|\le r\}$ as the intersection of half-planes, by Corollary~\ref{CR_lambda}\,(A) we  have
\begin{equation}\label{EQ_Re-est}
 \Re\big(-e^{-i\theta}G'(\tau,t)\big)\le(1+\cos\theta)\,\Big(\sum_{k=1}^n\frac{1}{G'(\sigma_k,t)}\,\Big)^{-1}
\end{equation}
for every $\theta\in[0,2\pi]$ and a.e.~$t\in[0,T]$

Denote $Q(x_1,x_2,\ldots,x_n):=\big(\sum_{j=1}^nx_j^{-1}\big)^{-1}$. From
\eqref{EQ_CP-rep1}, \eqref{EQ_CP-rep2}, and~\eqref{EQ_Re-est} we obtain
\begin{align}\label{EQ_CP-main}
\hskip5em\displaystyle%
&\hskip-5em\Re\big(-e^{-i\theta}\log\varphi'(\tau)\big)~=~\displaystyle\int_0^T\!\Re\big(-e^{-i\theta}G'(\tau,t)\big)\di t~\notag\\[1ex]
&\le~\mathrlap{\displaystyle(1+\cos \theta)\int_0^T\! Q\big(G'(\sigma_1,t),G'(\sigma_2,t),\ldots,G'(\sigma_n,t)\big)\,\di t} %
   \notag\\[1ex]
&\le~\textstyle(1+\cos \theta)\,T Q\Big(\frac1T\!\int_0^T\!G'(\sigma_1,t)\di t, \frac1T\!\int_0^T\!G'(\sigma_2,t)\di t,\ldots,\frac1T\!\int_0^T\!G'(\sigma_n,t)\di t\Big)%
    \notag\\[.7ex]
&=~\textstyle(1+\cos \theta)\, Q\Big(\int_0^T\!G'(\sigma_1,t)\di t, \int_0^T\!G'(\sigma_2,t)\di t, \ldots,\int_0^T\!G'(\sigma_n,t)\di t\Big)\notag\\[.7ex]
&=~\displaystyle(1+\cos \theta)\,\Big(\sum\limits_{k=1}^n\dfrac1{\log a_k}\Big)^{-1},
\end{align}
where we have taken into account that~$Q$ is a concave function on~$(0,+\infty)^n$, see Lemma~\ref{LM_concave} in the Appendix, and used Jensen's inequality, see
e.g.~\cite[p.\,76]{Ferguson}.

Note also that~$\varphi\neq\id_\UD$ and hence $\log\varphi'(\tau)\neq0$. Inequality~\eqref{EQ_CP-main} along with the latter remark shows that $U_\tau(F,A)\subset D(A)\setminus\{0\}$.
To see that $D(A)\setminus\{0\}\subset U_\tau(F,A)$, we apply Theorem~\ref{TH_param-repres}\,(B) with the autonomous Herglotz vector field
\begin{equation}\label{EQ_CP-G}
 G(z,t):=\frac{(\tau-z)(1-\overline\tau z)}{\displaystyle c\,+\,\sum_{k=1}^n \frac{|\tau-\sigma_k|^2}{2|\lambda_k|}\,\frac{\sigma_k+z}{\sigma_k-z}},
 \quad z\in\UD,~ t\in[0,T],
\end{equation}
where $\lambda_k:=-\frac1T\log a_k$, $k=1,\ldots,n$, and $c\in\Complex$ is an arbitrary  constant with $\Re c\ge0$. By Theorem~\ref{TH_representation-formula}, $G(\cdot,t)\in\GenDWe(F)$ for all $t\in[0,T]$. In this way we construct  $\varphi=\varphi_T\in\Ut\tau[F]$ such that
$\varphi'(\sigma_k)=a_k$ for $k=1,\ldots,n$ and
$$
 -\log\varphi'(\tau)=\left(\tilde c+\sum_{k=1}^{n}\frac1{2\log a_k}\right)^{\!-1}\!\!\!,\quad \tilde c:=\frac1{1-|\tau|^2}\left(\frac{c}{T}+i\sum_{k=1}^n\frac{\Im(\overline\sigma_k\tau)}{\log a_k}\right)\!.
$$
 Therefore, any value of $-\log\varphi'(\tau)$ from $D(A)\setminus\{0\}$ can be achieved by choosing a suitable~$c$ with $\Re c\ge0$.

It remains to study the case when $-\log\varphi'(\tau)=:\omega\in\partial U_\tau(F,A)\setminus\{0\}$, which takes place if and only if equality occurs in~\eqref{EQ_CP-main} for~$\theta=2\arg \omega$. In particular,  we should have equality for a.e.~$t\in[0,T]$ in~\eqref{EQ_Re-est}. By Corollary~\ref{CR_lambda}\,(A), this is possible only when $G(\cdot,t)$ is one of the functions~\eqref{EQ_G-zeta}. Recall also that $G(\cdot,t)\in\GenDWe(F)$, $t\in[0,T]$. Hence, for all~$z\in\UD$ and a.e. $t\in[0,T]$, $G(z,t)$ is given by~\eqref{EQ_CP-G} with functions $\lambda_k=\lambda_k(t)=-G'(\sigma_k,t)$ satisfying $\sum_{k=1}^n\lambda_k(t)=-1$ and with a suitable purely imaginary constant $c$, determined uniquely by the values of~$\omega$ and $\lambda_k$'s.

Moreover, using~\eqref{EQ_CP-rep2} and taking into account
that according to Lemma~\ref{LM_strict-concave}, the function~$Q$ is strictly concave on the set
$$
    \big\{{(x_1,\ldots,x_n)\in\R^n\colon} {x_k>0,}~{k=1,\ldots,n,\,}~{\,x_1+x_2+\ldots+x_n=1}\big\},
$$
we see that equality in~\eqref{EQ_CP-main} is only possible if $\lambda_k$'s are constants, i.e. $\lambda_k(t)=-\frac1T\log a_k$ for a.e.~$t\in[0,T]$ and all~$k=1,\ldots,n$. As a result, upon rescaling time in ODE~\eqref{EQ_param-repres},  the infinitesimal generator~$G_\omega$ appears in the right-hand side, which allows us to conclude that $\varphi=\phi^\omega_1$.
\end{proof}

Let us now consider the case of the boundary DW-point~$\tau$.

\begin{theorem}[\protect{\cite[Theorem~6.1]{CowenPommerenke}}]\label{TH_C-P_boundary}
Fix $\tau\in\UC\setminus F$ {and $A:=(a_k)_{k=1}^n\subset(1,+\infty)^n$.} 
The value region
    $$U_\tau(F,A):=\{-\log\varphi'(\tau)\in\Real\colon\varphi\in\Ut\tau[F],~\varphi'(\sigma_k)=a_k~\text{\rm for each}~k=1,\ldots,n\}$$
is the interval $\big[0,r(A)\big]$, where $r(A)$ is defined as in Theorem~\ref{TH_C-P}.

The equality $-\log\varphi'(\tau)=r(A)$ is achieved for a unique mapping $\varphi$, which coincides with the element~$\phi_1$ of the one-parameter semigroup~$(\phi_t)$ associated with the infinitesimal generator
\begin{equation}\label{EQ_C-P_boundary-equality}
 G(z):=(\tau-z)(1-\overline\tau z)\left(\displaystyle i\gamma\,+\,\sum\limits_{k=1}^n\dfrac{|\tau-\sigma_k|^2}{2\log a_k}\,\dfrac{\sigma_k+z}{\sigma_k-z}\right)^{\!-1}\!\!,\quad \gamma:=-\sum\limits_{k=1}^n\dfrac{\Im(\overline\sigma_k\tau)}{\log a_k},
\end{equation}
for all $z\in\UD$.
\end{theorem}
\begin{proof}
The proof uses the same ideas as in case~$\tau\in\UD$. Therefore, we only indicate the main differences without repeating all the details.

Clearly, $-\log \varphi'(\tau)\ge 0$.
To see that $-\log\varphi'(\tau)\le r(A)$, we argue as in the proof of Theorem~\ref{TH_C-P}, except that instead of~\eqref{EQ_Re-est} we should use the inequality
$$
-G'(\tau,t)\le\Big(\sum_{k=1}^n\frac{1}{G'(\sigma_k,t)}\,\Big)^{-1}
$$
from Corollary~\ref{CR_lambda}\,(B).

To identify the function~$\varphi$ for which the equality $-\log\varphi'(\tau)= r(A)$ is achieved, we again argue as in the proof of Theorem~\ref{TH_C-P}, arriving thus to the conclusion that up to rescaling of time, for a.e. $t\in[0,T]$, $G(\cdot,t)$ should be of the form given in Corollary~\ref{CR_lambda}\,(A${}'$) with $\lambda_k:=-\frac1T\log a_k$ for $k=1,\ldots, n$.

An example of~$\varphi\in\Ut\tau[F]$ with $-\log\varphi'(\tau)=0$ and $\varphi'(\sigma_k)=a_k$ for $k=1,\ldots,n$ is easily obtained from the infinitesimal generator~\eqref{EQ_C-P_boundary-equality} if we take any other real value of~$\gamma$.

All the remaining values in~$\big(0,r(A)\big)$ are delivered, e.g., by the elements~$\phi_1$ of the one-parameter semigroups generated by convex combinations of the infinitesimal generators corresponding to the values $0$ and~$r(A)$.
\end{proof}

\section{Appendix}
For completeness, we give proofs of some elementary facts used in the paper.
\begin{proof}[{\fontseries{bx}\selectfont Proof of Lemma \ref{LM_log}}]
Let $\varphi\in \Ut\tau[F]$. Fix some $\sigma_j,\sigma_k\in F$, $j\neq k$. Choose any $C^1$-smooth Jordan arc~$\Gamma\subset\overline\UD\setminus\{\tau\}$ joining $\sigma_j$ with $\sigma_k$  and orthogonal to~$\UC$ at these points.

Note that $\widetilde\Gamma:=\varphi(\Gamma)$ satisfies the same requirements imposed on~$\Gamma$. Let us show that $\Gamma$ and $\widetilde\Gamma$ are homotopic relative to end-points in~$\C\setminus\{\tau\}$.
Denote by $D_1$ and $D_2$ the two connected components of~$\UD\setminus\Gamma$, with $\tau\in D_2$, and let $C_1$ and $C_2$ be the two complementary arc of~$\UC$ such that $C_j\subset\partial D_j$, $j=1,2$.
In particular, $\Gamma$ is homotopic in~$\C\setminus\{\tau\}$ relative to the end-points to~$C_1$.

Furthermore, let $\widetilde D_1$ and $\widetilde D_2$ stand for the two connected components of ${\UD\setminus\widetilde\Gamma}$, numbered in such a way that $C_j\subset\partial\widetilde D_j$, $j=1,2$.
Using the conformality at~$\sigma_j$ of~$\varphi$ restricted to a Stolz angle, we see that the $\varphi(D_1)$ intersects~$\widetilde D_1$, and  $\varphi(D_2)$ intersects~$\widetilde D_2$. Since $\varphi:\UD\to\UD$ is a homemorphism onto its image, it follows that $\varphi(D_j)\subset\widetilde D_j$, $j=1,2$. In particular, $\tau=\varphi(\tau)\in\varphi(D_2)\subset \widetilde D_2$ and hence $\tau\not\in\widetilde D_1$. It follows that relative to end-points $\widetilde\Gamma$ is homotopic in~$\C\setminus\{\tau\}$ to~$C_1$ and hence to~$\Gamma$.
Therefore, for $\Phi(z):=(\varphi(z)-\tau)/(z-\tau)$ we have
$$
 \int_\Gamma\frac{\Phi'(z)}{\Phi(z)}\di z=\int_{\widetilde\Gamma}\frac{\di w}{w-\tau}~-~\int_{\Gamma}\frac{\di z}{z-\tau}~=~0.
$$
All the above integrals exist because $\tau\not\in\Gamma\cup\widetilde\Gamma$   and because $\varphi$ is of class~$C^1$ on $\Gamma$ including the end-points.

The above argument is valid for any two distinct points $\sigma_j,\sigma_k\in F$. Taking into account that $\Phi$ is holomorphic and non-vanishing in~$\UD$, it follows that $\Phi'/\Phi$ admits an antiderivative in~$\UD$ that has vanishing
angular limits at every point of~$F$, and this is the desired single-valued branch of~$\log \Phi$. This proves part~(A).

To prove part~(B),
note that continuity of $t\mapsto\varphi_t(z)$ for each~$z\in\UD$ is equivalent to continuity of~$t\mapsto\varphi_t\in\Hol(\UD,\Complex)$ in the open-compact topology because all holomorphic self-maps of~$\UD$ form a normal family. Therefore, for any $t_0\in I$ the limit
$$
 \lim_{I\ni t\to t_0}\exp F_t,\quad \text{where~}F_t:=\Psi[\varphi_t]-\Psi[\varphi_{t_0}],
$$
exists in the open-compact topology and equals~$1$ identically in~$\UD$.
Moreover, note that
\begin{equation}\label{EQ_F_t-anglimzero}
  \lim_{r\to 1^-}  F_t(r\sigma_{k_0})=0\quad\text{for all~$t\in I$}.
\end{equation}

Since by the hypothesis, for some $\delta>0$, the function ${t\mapsto \varphi'_t(\sigma_{k_0})}$ is bounded on $I_\delta:=I\cap(t_0-\delta,t_0+\delta)$, using Julia's Lemma~\ref{EQ-JuliasLemma} we see that for any $\varepsilon>0$ there exists $r_\varepsilon\in(0,1)$ such that $\gamma_\varepsilon:=[r_\varepsilon\sigma_{k_0},\sigma_{k_0})$ and $\varphi_t(\gamma_\varepsilon)$ lie in the disk $D_\varepsilon:=\{z:|z-(1-\varepsilon)\sigma_{k_0}|\le\varepsilon\}\subset\UD$ for all~$t\in I_\delta$.

Clearly we can choose $\varepsilon>0$ small enough, so that  $\tau\not\in D_\varepsilon$ and
\begin{equation}\label{EQ_Deps}
\max_{z,w\in D_\varepsilon}\big|\log(z-\tau)-\log(w-\tau)\big|=:C<\pi
\end{equation}
for some (and hence any) choice of the single-valued branch of $\log(z-\tau)$ in~$D_\varepsilon$.

Combining \eqref{EQ_F_t-anglimzero} and \eqref{EQ_Deps}, we see that $|F_t(r_\varepsilon\sigma_{k_0})|\le 2C<2\pi$ for all $t\in I_\delta$. Recalling that $\exp F_t(z)\to1$ locally uniformly in~$\UD$ as $I\ni t\to t_0$, we conclude that $F_t\to 0$ locally uniformly in~$\UD$ as $I\ni t\to t_0$,
which completes the proof of~(B).
\end{proof}

{Next lemma shows that the harmonic mean is concave. We include its proof for the sake of completeness.}
\begin{lemma}\label{LM_concave}
For any $n\in\Natural$ the function $Q(x_1,\ldots x_n):=\Big(\sum_{j=1}^n x_j^{-1}\Big)^{-1}$ is concave on~$(0,+\infty)^n$.
\end{lemma}
\begin{proof}
The assertion of the lemma holds trivially for $n=1$. So we suppose that $n\ge2$.

The entries of the Hessian matrix $A(\mathbf x)=[a_{jk}(\mathbf x)]$, $\mathbf x:=(x_1,x_2,\ldots,x_n)$, for the function~$Q$ are given by
$$
 a_{jk}(\mathbf x)=\frac{2Q(\mathbf x)^3}{x_j^2x_k^2}b_{jk}(\mathbf x),\quad \text{where~}b_{jk}(\mathbf x):=1-\delta_{jk}\frac{x_j}{Q(\mathbf x)}~\text{~and~$\delta_{jk}$ is the Kronecker symbol.}
$$
First we show that $\det A(\mathbf x)=0$. Clearly, the latter is equivalent to $\det B(\mathbf x)=0$, where $B(\mathbf x):=[b_{jk}(\mathbf x)]$.

Subtract the last row of $B(\mathbf x)$ from each of the other rows. In the matrix we obtain, add to the last row the linear combination of all the other rows in which, for every $j=1,2,\ldots,n-1$, the coefficient of the $j$-th row is equal to $Q(\mathbf x)/x_j$. The resulting matrix is upper-triangular, with the last diagonal entry equal to
$$
 1-\frac{x_n}{Q(\mathbf x)}~+~\frac{x_n}{Q(\mathbf x)}\sum_{j=1}^{n-1}\frac{Q(\mathbf x)}{x_j}~=~1~-~\sum_{j=1}^n \frac{x_n}{x_j}~+~\sum_{j=1}^{n-1} \frac{x_n}{x_j}~=~0.
$$
Therefore, the determinant equals zero.

This argument, with an obvious modification, can be used to show that the determinants of all symmetric minors of~$A(\mathbf x)$, i.e. minors of the form $[a_{jk}(\mathbf x)]_{j,k\in J}$, $J\subset\{1,2,\ldots,n\}$, vanish, except for the symmetric minors of order one. They are simply diagonal entries of~$A$, which are all negative. Therefore, according to Sylvester's well-known criterion, the matrix~$A(\mathbf x)$ is negative semi-definite for any~$\mathbf x\in(0,+\infty)^n$, which was to be proved.
\end{proof}

\begin{lemma}\label{LM_strict-concave}
Let $Q$ be defined as in Lemma~\ref{LM_concave}. Let $\mathbf x,\mathbf y\in(0,+\infty)^n$, $\mathbf x\neq \mathbf y$.
If
\begin{equation}\label{EQ_not-strict}
\lambda Q(\mathbf x)+(1-\lambda)Q(\mathbf y)=Q\big(\lambda \mathbf x + (1-\lambda)\mathbf y\big)
\end{equation}
 for some~$\lambda\in(0,1)$, then $\mathbf x=\mu\mathbf y$ for some~$\mu\in\Real$.
\end{lemma}
\begin{proof}
Since by Lemma~\ref{LM_concave}, $Q$ is concave on~$(0,+\infty)$, equality~\eqref{EQ_not-strict} for some $\lambda\in(0,1)$ implies the same equality for all~$\lambda\in[0,1]$. Notice that the r.h.s. of~\eqref{EQ_not-strict}, $f(\lambda):=Q\big(\lambda \mathbf x + (1-\lambda)\mathbf y\big)$ is a rational function of~$\lambda$. Therefore, extending~$f$, as usual, to its removable singularities by continuity, we may conclude that $f(\lambda)=a\lambda+b$ for some $a,b\in\Real$ and all~$\lambda\in\Real$.
On the one hand, a function of this form has at most one zero. On the other hand, taking into account that for $\lambda=0$ all the components of the vector $\mathbf x_\lambda:=\lambda \mathbf x + (1-\lambda) \mathbf y$ are positive,  it is easy to see that $f(\lambda)=0$ for each $\lambda\in\Real$ such that at least one component of~$\mathbf x_\lambda$ vanishes.

All non-vanishing components of~${\Real\ni\lambda\mapsto\mathbf x_\lambda}$  are positive constants. Therefore, if such components exist, then $a=0$ and hence $f(\lambda)\equiv Q(\mathbf y)>0$. It follows that if at least one of the components is non-vanishing, then $f(\lambda)$ does not vanish, which means that, in fact, all the components of ${\Real\ni\lambda\mapsto\mathbf x_\lambda}$ are non-vanishing. This in turn would imply that $\mathbf x=\mathbf y$ in contradiction to the hypothesis.

Thus we may conclude that all the components of $\mathbf x_\lambda$ vanish for the same value of~$\lambda$ and the desired conclusion follows immediately.
\end{proof}

\end{document}